\documentclass[12pt,xcolor=svgnames]{amsart}

\usepackage[toc,page]{appendix}
\usepackage{amsmath, amsthm, amssymb, amsfonts, verbatim}

\usepackage{xspace,mathtools,bm}

\usepackage[svgnames]{xcolor}

\usepackage[pagebackref=true, colorlinks=true, citecolor=blue]{hyperref}

\usepackage{graphicx} 
\usepackage{amsmath}
\usepackage{amssymb}
\usepackage{amsthm}
\usepackage{stackengine}
\usepackage{hyperref}
\hypersetup{
    colorlinks=true,
    citecolor=red,
    linkcolor=blue,
    filecolor=magenta,      
    urlcolor=cyan,
    pdfpagemode=FullScreen,
    }
\newtheorem{theorem}{Theorem}[section]

\newtheorem{lemma}[theorem]{Lemma}
\newtheorem{definition}[theorem]{Definition}
\newtheorem{proposition}[theorem]{Proposition}
\newtheorem{remark}[theorem]{Remark}

\newcommand{\R}{\mathbb{R}}

\newcommand{\eps}{\varepsilon}
\newcommand{\abs}[1]{\mid\!#1\!\mid}
\newcommand{\norm}[2]{\left\lVert#1\right\rVert_{#2}}
\newcommand{\eqdef}{:=}

\DeclareMathOperator\supp{supp}
\usepackage{mathtools}

\DeclarePairedDelimiter\floor{\lfloor}{\rfloor}

\title[Non-Decaying Solutions to generalized SQG]{Existence of Non-Decaying Solutions to the Generalized Surface Quasi-Geostrophic Equations}
\author{Zachary Radke }
\date{\today}

\begin{document}

\maketitle

\begin{abstract}
We establish the short-time existence and uniqueness of non-decaying solutions to the generalized Surface Quasi-Geostrophic equations in Hölder-Zygmund spaces $C^r(\R^2)$ for $r>1$ and uniformly local Sobolev spaces $H_{ul}^s(\R^2)$ for $s>2$.
\end{abstract}

\section{Introduction}
\subsection{Background}
We study non-decaying solutions of the generalized surface quasi-geostrophic equations. This model arises as an interpolation between the 2D Euler and surface quasi-geostrophic equations. The 2D Euler equations are
\begin{equation}
\label{Euler Velocity}
\tag{$E$}
\begin{cases}
    \partial_t u + u\cdot \nabla u + \nabla p=0 \hspace{.5cm} &\text{in } [0,T]\times\R^2,\\
    \text{div }u=0 \hspace{.5cm} &\text{in } [0,T]\times\R^2, \\
    u\vert_{t=0}=u^0 \hspace{.5cm} &\text{in } \R^2,
\end{cases}
\end{equation}
where $u:[0,T]\times \R^2 \to \R^2$ is the velocity and $p:[0,T]\times\R^2 \to \R$ is the scalar pressure. In vorticity form the 2D Euler equations are
\begin{equation}
\label{Euler Vorticity}
\tag{$E_{\omega}$}
\begin{cases}
\partial_t \omega + u \cdot \nabla \omega = 0 \hspace{.5cm}& \text{in } [0,T] \times \R^2, \\
u = \mathcal{K} \ast \omega = \nabla^{\perp}(-\Delta)^{-1}\omega \hspace{.5cm} &\text{in } [0,T] \times \R^2, \\
\omega \vert_{t=0} = \omega^0 \hspace{.5cm} &\text{in } \R^2,
\end{cases}
\end{equation} 
where $\omega = \partial_1u_2-\partial_2u_1$. The surface quasi-geostrophic equations can be written as
\begin{equation}
\label{SQG}
\tag{$SQG$}
\begin{cases}
\partial_t \theta + u \cdot \nabla \theta = 0 \hspace{.5cm} &\text{in } [0,T]\times \R^2, \\
u=\nabla^{\perp}(-\Delta)^{-1/2} \theta \hspace{.5cm} &\text{in } [0,T]\times \R^2, \\
\theta\vert_{t=0} = \theta^0 \hspace{.5cm} &\text{in } \R^2.
\end{cases}
\end{equation} 
We see that \eqref{Euler Vorticity} and \eqref{SQG} are very similar in that the scalar quantity of interest is being transported by a divergence-free vector field, but differ in that the velocity $u$ in $\eqref{Euler Vorticity}$ gains one more spatial degree of regularity than $\omega$. This is in contrast to \eqref{SQG} where the velocity and scalar have roughly the same spatial regularity.
\\
\\
In this paper we interpolate between $\eqref{Euler Vorticity}$ and $\eqref{SQG}$ by defining the velocity as $u = \nabla^{\perp}(-\Delta)^{-1+\beta/2}$ for $\beta \in (0,1)$. Here $(-\Delta)^{\frac{s}{2}}$ is the fractional Laplacian which is defined as 
$$
(-\Delta)^{\frac{s}{2}}f(x) = C_{d,s}PV\int_{\R^2}\frac{f(x)-f(y)}{|x-y|^{d+s}}dy,
$$
for $s>0$ and $f \in \mathcal{S}'(\R^n)$. This gives the generalized SQG equation
\begin{equation}
\label{gSQG}
\tag{$gSQG_{\beta}$}
\begin{cases}
\partial_t \theta + u \cdot \nabla \theta = 0 \hspace{.5cm} &\text{in } [0,T]\times \R^2, \\
u=\nabla^{\perp}(-\Delta)^{-1+\beta/2} \theta \hspace{.5cm} &\text{in } [0,T]\times \R^2, \\
\theta\vert_{t=0} = \theta^0 \hspace{.5cm} &\text{in } \R^2.
\end{cases}
\end{equation}
Note that when $\beta=0$ we recover \eqref{Euler Vorticity} and when $\beta=1$ we recover \eqref{SQG}. We also see that we can write the velocity as
$$
u=\nabla^{\perp} (-\Delta)^{-1}(-\Delta)^{\frac{\beta}{2}}\theta = \mathcal{R}^{\perp}(-\Delta)^{\frac{\beta-1}{2}}\theta,
$$
where $\mathcal{R}$ is the Riesz transform defined as $$\mathcal{R}f = \nabla (-\Delta)^{-\frac{1}{2}} f.$$
Moreover, one can write $u=\nabla^{\perp}(\Phi_{\beta}\ast \theta)$ where $\Phi_{\beta}$ is the fundamental solution of $(-\Delta)^{1-\frac{\beta}{2}}$ defined as 
$$
\Phi_{\beta}(x)=\frac{C_{\beta}}{\abs{x}^{\beta}}.
$$ In a rough sense, we see that the velocity $u$ in \eqref{gSQG} gains $1-\beta$ degrees of spatial regularity more than the transported quantity $\theta$.

\subsection{Main Results}
We state our main result below in condensed form. More complete statements can be found in Theorema \ref{Holder Result} and Theorem \ref{Hsul Result}. Respectively, see Definitions \ref{Hsul definition}, \ref{zygmund definition}, and \ref{vector convolution definition} for the definitions of $H^s_{ul}$, $C^r$, $\dot{C}^r$ and $v\ast\cdot w$.
 \begin{theorem}
Let $\beta \in (0,1)$. Assume that $\theta^0 \in C^r(\R^2)$ for some $r>1$ and $u^0 \in L^{\infty}(\R^2)$ satisfy $$u^0 = \nabla^{\perp}(-\Delta)^{-1+\beta/2}\theta^0 \text{ in } \dot{C}^r(\R^2).$$ Then there exists some $T>0$ and a unique solution $(u,\theta)$ to \eqref{gSQG} with the constitutive law $u=\nabla^{\perp}(-\Delta)^{-1+\beta/2}$ replaced by the Serfati-type identity
$$
u(t)=u^0 + \nabla^{\perp}(a\Phi_{\beta}) \ast (\theta(t)-\theta^0) -\int_0^t (\nabla\nabla^{\perp}((1-a)\Phi_{\beta})\ast\cdot(\theta u)(\tau))d\tau,
$$
satisfying for any $r'\in (0,r)$,
\begin{align*}
\theta &\in L^{\infty}([0,T]; C^r(\R^2)) \cap \text{Lip}([0,T]; C^{r-1}(\R^2)) \cap C([0,T]; C^{r'}(\R^2)),\\
u &\in L^{\infty}([0,T]; C^{r+1-\beta}(\R^2)) \cap C([0,T]; C^{r'+1-\beta}(\R^2)).
\end{align*}
\end{theorem}
\begin{theorem}
Let $\beta \in (0,1)$. Assume that $\theta^0 \in H_{ul}^s(\R^2)$ and $u^0 \in H_{ul}^{s+1-\beta}(\R^2)$ for some $s>2$ satisfy 
 $$u^0 = \nabla^{\perp}(-\Delta)^{-1+\beta/2}\theta^0 \text{ in } \dot{C}^{\alpha}(\R^2),$$ where $\alpha \in (1,s-1)$.
 Then there exists some $T>0$ and a unique solution $(u,\theta)$ to \eqref{gSQG} with the constitutive law $u=\nabla^{\perp}(-\Delta)^{-1+\beta/2}$ replaced by the Serfati-type identity
$$
u(t)=u^0 + \nabla^{\perp}(a\Phi_{\beta}) \ast (\theta(t)-\theta^0) -\int_0^t (\nabla\nabla^{\perp}((1-a)\Phi_{\beta})\ast\cdot(\theta u)(\tau))d\tau,
$$
satisfying
 \begin{align*}
\theta &\in L^{\infty}([0,T];H_{ul}^s(\R^2)) \cap Lip([0,T];H_{ul}^{s-1}(\R^2)) \\
u &\in L^{\infty}([0,T]; H_{ul}^{s+1-\beta}(\R^2)). 
\end{align*}
 \end{theorem}

\subsection{Known Results}
Global existence of smooth solutions to \eqref{SQG} is a very challenging open question.  Since this question is so difficult, one idea was to regularize the velocity field in some manner. One of the earliest papers to take this approach is \cite{10}, where the authors study a logarithmic generalization of the velocity given by $$u = \nabla^{\perp}(-\Delta)^{-1}(\log(1+\log(1-\Delta)))^{\beta}\theta,$$ which is more singular than the velocity in \eqref{Euler Velocity}. For such a model, the authors establish global well-posedness in $L^1 \cap L^{\infty} \cap B_{q,\infty}^s(\R^2) $. Furthermore, they provide a blow-up criterion. In particular, they assume for $\beta \in [0,1],\theta_0 \in C^{\sigma}(\R^2) \cap L^q(\R^2)$ with $\sigma>1$, $q>1$, that if $\theta$ satisfies, for $T>0$:
$$
\int_0^T \norm{\theta(\cdot,t)}{C^{\beta}}dt < \infty,
$$
then $\theta$ remains in $C^{\sigma}(\R^2) \cap L^q(\R^2)$ on the time interval $[0,T]$.
\\
\\
 It is also natural to ask what happens the velocity $u$ is more singular than the scalar $\theta$. This problem is investigated in \cite{10} where the authors study the generalization of the velocity field given by $u=\nabla^{\perp}(-\Delta)^{-\frac{1+\beta}{2}}$ where $\beta \in (1,2)$. The authors establish local well-posedness for \eqref{gSQG} given $\theta_0 \in H^m(\R^2)$ for $m \geq 4$ an integer. Existence for global weak solutions on the torus $\mathbb{T}^2$ is shown provided that $\theta$ is mean zero, i.e. $\int_{\mathbb{T}^2}\theta dx =0$.
\\
\\
There has also been  work done for the ill-posedness for \eqref{gSQG} when $\beta \in (1,2)$. Very recently, in \cite{9}, Córdoba and Martínez-Zoroa establish strong ill-posedness in classical Hölder Spaces. In this setting, for $k \geq 2$, the authors also construct solutions with initial data in $C^{k,\gamma}(\R^2)\cap L^2(\R^2)$ that instantaneously leave $C^{k,\gamma}(\R^2)$, but lie in some Sobolev space of sufficiently small order. 
\\
\\
For the study of non-decaying solutions to fluid equations one of the earliest results can be found in \cite{7} where Serfati studies solutions to \eqref{Euler Velocity} with initial velocity and vorticity in $L^{\infty}(\R^2)$ and potentially lacking spatial decay. One of the main obstacles in this setting is that the constitutive law typically ceases to make sense. We thus adapt the constitutive law to hold in a weaker sense. This is often achieved by applying a cut-off function to the kernel $\mathcal{K}$ to split the kernel into a near-field and far-field term. Then integrating the far-field term by parts yields a constitutive law which makes sense for functions without decay. Such an identity is often referred to as the Serfati identity \cite{7}.
\\
\\
For more regular solutions than those considered in \cite{7}, the authors in \cite{14} and \cite{15} establish global well posedness to the 2D Navier Stokes equations with initial velocity in $L^{\infty}(\R^2)$. The Boussinesq equations have also been studied in such spaces, see for example \cite{16}. In \cite{5} Wu studies solutions to \eqref{SQG} in Hölder spaces $C^{k,\gamma}(\R^2)$, but assumes some integrability condition on $\theta$ so the constitutive law makes sense. In \cite{1} the authors remove the integrability conditions and replace the constitutive law with a Serfati-type identity. In \cite{1}, the authors establish well-posedness to both \eqref{Euler Velocity} and \eqref{SQG} in Hölder-Zygmund and uniformly local Sobolev spaces $H_{ul}^m(\R^2)$ for $m \geq 3$ an integer.\\
\\
In this work, we expand upon \cite{1} by establishing a similar result for \eqref{gSQG} in Hölder-Zygmund spaces \( C^r(\mathbb{R}^2) \) and uniformly local Sobolev spaces \( H^s_{ul}(\mathbb{R}^2) \) for \( s > 2 \), where \( s \) is not necessarily an integer. Since fractional derivatives are nonlocal, working in \( H^s_{ul}(\mathbb{R}^2) \) introduces significant technical challenges. This is due to the fact that we aim to derive local estimates for a nonlocal operator. In particular, the a priori estimate required in \( H^s_{ul}(\mathbb{R}^2) \) gives rise to several new difficulties. The main one is controlling the commutator involving the Bessel potential. Another involves obtaining precise estimates for the near-field behavior of the constitutive law.
\\
\\
\subsection{Notation and Organization of Paper}
We next define some of the notations that we will use throughout this paper. Consider two quantities $A,B$ parameterized by some index set $\Lambda$, we then say that $A \lesssim B$ if there exists some $C>0$ such that $A(\lambda) \leq CB(\lambda)$ for all $\lambda \in \Lambda$. If we wish to highlight the dependence of $C$ on a parameter, say $b$, we write $A \lesssim_b B$.  We write $A \sim B$ if $A \lesssim B$ and $B \lesssim A$
\\
\\
We let $a:\R^2 \to \R$ denote a radially symmetric bump function such that $a$ is identically $1$ in $B_1(0)$ and is monotone decreasing for $1 \leq \abs{x} \leq 2$, $a \in C_0^{\infty}(\R^2)$, and $\text{supp } a = \{x \in \R^2: \abs{x} \leq 2\}$. 
\begin{definition}
\label{vector convolution definition}Given two vector fields $v,w$ we define $v \ast \cdot w = \sum_i v^i \ast \omega^i$. Similarly, for matrix-valued functions $A$ and $B$ on $\R^d$, we define $(A\ast\cdot B)^{ij}=A^{ij}\ast B^{ij}$.
\end{definition}
The paper will be organized as follows. In Section 2, we start with a brief introduction to Littlewood-Paley theory and useful function spaces. Following, in Section 3 we show short-time existence in uniformly local Sobolev spaces. We do this by taking a sequence of approximating solutions in Hölder Zygmund spaces. We note that the proof in Hölder Zygmund spaces is identical to that in \cite{1}, one just needs to track the indices. For sake of completeness, we outline the well-posedness proof of Hölder Zygmund spaces in the appendix.
\section{Preliminaries and Lemmas}

\subsection{ODE Lemmas} We make extensive use of Grönwall's Lemma which we state below.
\begin{lemma}
Let $T>0$ and $u,\alpha,\beta$ be nonnegative, continuous functions on $[0,T]$ with $\alpha$ nondecreasing. If for all $t\in[0,T],$
$$
u(t) \leq \alpha(t)+\int_0^t\beta(s)u(s)ds,
$$
then for all $t\in [0,T]$,
$$
u(t)\leq\alpha(t)\exp\left(\int_0^t \beta(s)ds\right).
$$
\end{lemma}
\subsection{Littlewood-Paley Theory}
Before defining our function spaces, we begin with a brief overview of the Littlewood-Paley operators. We closely follow \cite{3}.\\
\\
It is well known that if $\mathcal{C}$ is the annulus of center $0$, short radius $3/5$, and long radius $5/3$, then there exists two positive radial functions $\hat{\chi},\hat{\varphi} \in C_0^{\infty}(B_{5/6}(0))$ and $C_0^{\infty}(\mathcal{C})$ respectively such that, if for $j \in \mathbb{Z}$ we set $\varphi_j(x)=2^{jd}\varphi(2^jx)$, then
$$
\hat{\chi}+\sum_{j\geq 0} \hat{\varphi}_j = \hat{\chi} + \sum_{j \geq 0} \hat{\varphi}(2^{-j}\cdot)
 \equiv 1.$$
For $n \in \mathbb{Z}$ we define $\chi_n$ via its Fourier transform as
$$
\hat{\chi}_n(\xi) = \hat{\chi}(\xi) + \sum_{j \leq n} \hat{\varphi}_j(\xi) \text{ for } \xi \in \R^d.
$$
For  $f \in \mathcal{S}'(\R^d)$, we define the operator $S_n$ as 
$$
S_n f =\chi_n \ast f.
$$
Formally, $S_n$ can be thought of as a projection in Fourier space into the ball of radius of order $2^n$. For $f \in \mathcal{S}'(\R^d)$ we define the inhomogeneous Littlewood-Paley operators $\Delta_j$ by 
\begin{equation*}
\label{LP operator}
\Delta_j f
=
\begin{cases}
    0, \hspace{1cm} \text{if }j < -1\\
    \chi \ast f, \hspace{1cm} \text{if } j=-1 \\
    \varphi_j \ast f, \hspace{1cm} \text{if } j\geq 0,
\end{cases}
\end{equation*}
and for all $j\in \mathbb{Z}$ we define the homogeneous Littlewood-Paley operators $\dot{\Delta}_j$ by
$$
\dot{\Delta}_j f = \varphi_j \ast f.
$$
One can formally think of $\dot{\Delta}_j$ as a projection in Fourier space into the annulus of inner and outer radius of order $2^j$. For any $u \in \mathcal{S}'(\R^d)$, we have that $u = \lim_{n \to \infty} S_nu$ where the limit holds in $\mathcal{S}'(\R^d)$; equivalently
$$
u = \sum_{j\geq -1} \Delta_j u \text{ in } \mathcal{S}'(\R^d).
$$
One benefit of using Littlewood-Paley operators can be seen in Bernstein's Lemma. Heuristically, Bernstein's Lemma states that if the support of $\hat{u}$ is included in some ball $B_{\lambda}(0)$, then a derivative of $u$ costs no more than $\lambda$. This will be made precise in Lemma \ref{Bernstein's Lemma} below. We let $\mathcal{C}_{a,b}(0)$ denote the annulus with inner radius $a$ and outer radius $b$.

\begin{lemma}
\label{Bernstein's Lemma}(Bernstein's Lemma) Let $(r_1,r_2)$ be a pair of strictly positive real numbers such that $r_1<r_2$. Then there exists a constant $C>0$, such that, for every integer $k$, every pair of real numbers $(a,b)$ satisfying $b \geq a \geq 1$, and every function $u \in L^a(\R^d)$,
\begin{enumerate}

\item 
\label{Berstein in Ball}
If $\supp{\hat{u}} \subset B_{r_1\lambda}(0)$ then
\begin{equation*}
 \sup_{\abs{\alpha} = k} \norm{\partial^{\alpha}u}{L^b} \leq C^k\lambda^{k+d(\frac{1}{a}-\frac{1}{b})}\norm{u}{L^a}.
\end{equation*}\\
\item 
\label{Berstein in Annulus}
Furthermore, if $\supp{\hat{u}} \subset \mathcal{C}_{r_1\lambda, r_2\lambda}(0)$ then
\begin{equation*}
 C^{-k}\lambda^k\norm{u}{L^a} \leq  \sup_{\abs{\alpha} = k} \norm{\partial^{\alpha}u}{L^a} \leq C^k\lambda^k\norm{u}{L^a}.
\end{equation*}
\end{enumerate}
\end{lemma}

\subsection{Function Spaces}

We begin by defining the Bessel potential, allowing us to make sense of Sobolev spaces for non-integer valued exponents.
\begin{definition} Let $s\in \R$. We define the Bessel potential of order $s$ to be $$J^s=(I-\Delta)^{s/2},$$ where  
$$
(I-\Delta)^{s/2}f = \mathcal{F}^{-1}(\widehat{G}^s\widehat{f}) = G^s \ast f,
$$
and
$$
G^s(x)= C_n \mathcal{F}^{-1}\left(1+4\pi^2\abs{\xi}^2\right)^{s/2}(x).
$$
\end{definition}
We have the following commutator estimate on the Bessel potential due to Kato and Ponce \cite{6}.

\begin{proposition} Suppose that $s>0$ and $p\in (1,\infty)$. Then for $f,g \in \mathcal{S}(\R^d)$ we have 
\begin{equation*}
\label{bessel commutator}
\norm{J^s(fg)-fJ^sg}{L^p} \lesssim \norm{\nabla f}{L^{\infty}}\norm{J^{s-1}g}{L^p}+\norm{J^s f}{L^p}\norm{g}{L^{\infty}}.
\end{equation*}
\end{proposition}
We remark that Proposition \ref{bessel commutator} holds for functions in spaces of which the Schwartz class is dense. 
\begin{definition}
\label{definion of bessel potential spaces}Let $p \in (1,\infty)$. We define the Bessel Potential spaces $H_p^{s}(\R^d)$ and $\dot{H}_p^s(\R^d)$ to be \begin{align*}
H_p^s(\R^d)=\{f \in \mathcal{S}'(\R^d) : \norm{J^sf}{L^p} <\infty\},\\
\dot{H}_p^s(\R^d)=\{f \in \mathcal{S}'(\R^d) : \norm{(-\Delta)^{\frac{s}{2}}f}{L^p} <\infty\},
\end{align*}
equipped with the norms $$\norm{f}{H_p^s} = \norm{J^sf}{L^p} \text{ and }\norm{f}{\dot{H}_p^s} = \norm{(-\Delta)^{\frac{s}{2}}f}{L^p}.$$
For the case of $p=2$ we define $H^s(\R^d)=H_2^{s}(\R^d)$ and $\dot{H^s}(\R^d)=\dot{H}_2^{s}(\R^d)$. 
\end{definition}

\begin{remark}(Equivalent $H^s$ norms) We note that the above norm for $H^s(\R^d)$ is equivalent to the norm,  
$$\norm{f}{H^s} = \left( \sum_{j \geq -1} 2^{2js}\norm{\Delta_j f}{L^2}^2\right)^{1/2}.$$
A proof of equivalence can be found in \cite{3}.

\end{remark}
In what follows we let $\phi \in C_0^{\infty}(\R^2)$ be a standard bump function which is identically $1$ on $B_1(0)$, vanishes on $(B_2(0))^c$, and decreases monotonically when $1 \leq \abs{x}\leq 2$. We also define $\phi_x = \phi(x - \cdot)$ and $\phi_{x,\lambda}=\phi\left(\frac{x-\cdot}{\lambda}\right)$ for any $\lambda>0$.

\begin{definition}
\label{Hsul definition} We define the uniformly local Sobolev space of order $s>0$ as
$$
H_{ul}^s(\R^2) = \{ f \in \mathcal{S}'(\R^2) : \sup_{x\in \R^2} \norm{\phi_x f}{H^s} < \infty\},
$$
with norm 
\begin{equation}
\label{hsul norm}
\norm{f}{H_{ul}^s} = \sup_{x \in \R^2}\norm{\phi_xf}{H^s}.
\end{equation}
We can also define the homogeneous version, as in Definition \ref{definion of bessel potential spaces}.
\end{definition}

\begin{remark} For any fixed $\lambda>0$, we can define an equivalent norm on $H^s_{ul}(\R^2)$ as $$\norm{f}{H^s_{ul}}=\sup_{x \in \R^2}\norm{\phi_{x,\lambda}f}{H^s}.$$ A proof can be found in \cite{12}.
\end{remark}

We will use the following lemma which can be found in \cite{1} with its proof.
\begin{lemma}
\label{Lp norm of Snf}
Let $p \in [1,\infty)$ and $f \in L_{ul}^p(\R^d)$. Then 
$$
\norm{S_nf}{L_{ul}^p} \lesssim\norm{f}{L_{ul}^p},
$$
where the above implicit constant does not depend on $n$.
\end{lemma}

\begin{definition} Let $p \in [1,\infty)$, $s>0$ and write $s=m+\sigma$ where $m \in \mathbb{Z}_{\geq 0}$ and $\sigma \in (0,1)$. We define the Sobolev-Slobodeckij space $W^{s,p}(\R^2)$ as
$$
W^{s,p}(\R^2) = \left\{f \in W^{m,p}(\R^2) : \left[f \right]_{W^{s,p}(\R^2)} = \left(\iint_{\R^2 \times \R^2} \frac{\abs{\nabla^mf(x)-\nabla^mf(y)}^p}{\abs{x-y}^{2+\sigma p}} dx dy\right)^{\frac{1}{p}} <\infty \right\}. 
$$
We note that this is a Banach Space under the norm
$$
\norm{f}{W^{s,p}(\R^2)} = \norm{f}{W^{m,p}(\R^2)} + \left[f \right]_{W^{s,p}(\R^2)}. 
$$
We also define the uniformly local Sobolev-Slobodeckij space $W_{ul}^{s,p}(\R^2)$ via the norm 
\begin{equation}
\label{SSul norm}
\norm{f}{W_{ul}^{s,p}} = \sup_{x\in \R^2} \norm{\phi_xf}{W^{s,p}}.
\end{equation}

\end{definition}

\begin{proposition}
\label{SS equivalent norm}The norms $\norm {\cdot}{W_{ul}^{s,2}}$ and $\norm{\cdot}{H_{ul}^s}$ are equivalent.
\end{proposition}
\begin{proof}
In view of \eqref{hsul norm} and \eqref{SSul norm}, we must show that the $W^{s,2}(\R^2)$ and $H^s(\R^2)$ norms are equivalent. This is fairly standard so we only provide an outline of the proof. We begin by recalling that the fractional Laplacian for $(-\Delta)^{\sigma}u$ in $\R^2$  for $\sigma \in (0,1)$ is defined as 
$$
(-\Delta)^{\sigma}u(x) = C_{d,\sigma}\int_{\R^2} \frac{u(x)-u(y)}{\abs{x-y}^{2+2\sigma}}dy.
$$
Note by Plancherel's theorem
$$
\int_{\R^2} u (-\Delta)^{\sigma} u dx = \norm{u}{\dot{H}^{\sigma}}^2.
$$
But we also see that 
\begin{align*}
\int_{\R^2} u (-\Delta)^{\sigma} u dx &= \int_{\R^2} u(x) \left(C_{\sigma}\int_{\R^2} \frac{u(x)-u(y)}{\abs{x-y}^{2+2\sigma}}dy\right)dx\\&\hspace{1cm}+ \int_{\R^2} u(y) \left(C_{\sigma}\int_{\R^2} \frac{u(y)-u(x)}{\abs{x-y}^{2+2\sigma}}dx\right)dy\\
&=
\frac{C_{\sigma}}{2} \iint_{\R^2 \times \R^2} \frac{\abs{u(x)-u(y)}^2}{\abs{x-y}^{2+2\sigma}}dydx.
\end{align*}
We then note
\begin{equation}
\label{ss equivalent seminorm}
\norm{u}{\dot{H}^{\sigma}}^2 = \frac{C_{\sigma}}{2} \iint_{\R^2 \times \R^2} \frac{\abs{u(x)-u(y)}^2}{\abs{x-y}^{2+2\sigma}}dydx=[u]_{W^{\sigma,2}}^2.
\end{equation}
Observing that 
$$
\norm{u}{H^m}^2 = \norm{u}{L^2}^2 + \norm{\nabla^mu}{L^2}^2 = \int_{\R^2} (1+\abs{\xi}^{2m})\abs{\hat{u}(\xi)}^2d\xi,$$
and using \eqref{ss equivalent seminorm} allows us to conclude
\begin{align*}
\norm{u}{W^{s,2}}^2 &= \norm{u}{W^{m,2}}^2 + \norm{\nabla^mu}{\dot{H}^{\sigma}}^2 \\ 
&= \int_{\R^2} (1+\abs{\xi}^{2m}+\abs{\xi}^{2s})\abs{\hat{f}(\xi)}^2 d\xi.
\end{align*}
Then we note that there exists two positive constants $C_1,C_2>0$ such that 
$$
C_1 \leq \frac{1+\abs{\xi}^{2m}+\abs{\xi}^{2s}}{(1+\abs{\xi}^2)^s} \leq C_2,
$$
which shows that the two norms are equivalent.
\end{proof}

 \begin{definition}
 \label{zygmund definition} For $r \in \mathbb{R}$ we define the Hölder-Zygmund space $C^r(\R^d)$ to be the set of tempered distributions $f$ such that
 $$
 \norm{f}{C^r} = \sup_{j \geq -1} 2^{jr}\norm{\Delta_j f}{L^{\infty}}<\infty.
 $$
 We also define the homogeneous Hölder-Zygmund space $\dot{C}^r(\R^d)$ with norm
 $$
 \norm{f}{\dot{C}^r} = \sup_{j \in \mathbb{Z}}\norm{\dot{\Delta}_jf}{L^{\infty}}.
 $$
\end{definition}
\begin{definition}
For $m\geq 0$ an integer, we define the Hölder space $\tilde{C}^m(\R^d)$ as
$$
\tilde{C}^m(\R^d) = \{f: D^{\alpha}f \in \tilde{C}(\R^d) \text{ for all } |\alpha| \leq m \text{ and } \norm{f}{\tilde{C}^m}<\infty\},$$
where $\tilde{C}(\R^d)$ is the set of continuous functions on $\R^d$ and 
$$
\norm{f}{\tilde{C}^m}=\sum_{|\alpha|\leq m} \norm{D^{\alpha}f}{L^{\infty}}.
$$
It is important to note that here $D^{\alpha}$ are classical derivatives. If $r>0$ is not an integer then we define 
 $$
 \tilde{C}^r(\R^d)=\{ f \in \tilde{C}^{\floor{r}}(\R^d): \norm{f}{\tilde{C}^r}<\infty\},
 $$ where
 \end{definition}
 \begin{equation}
\norm{f}{\tilde{C}^r} = \sum_{0 \leq |\alpha|\leq \floor{r}} \norm{D^{\alpha}f}{L^{\infty}} + \sum_{|\beta|=\floor{r}} \sup_{x \neq y} \frac{\abs{D^{\beta}f(x)-D^{\beta}f(y)}}{\abs{x-y}^{r-\floor{r}}}.
 \end{equation}
 It is well known that when $r>0$ is a non-integer $C^r(\R^d) = \tilde{C}^r(\R^d)$ (see for example \cite{3}), but if $r$ is an integer we do not have equality; rather we have the inclusion
 $$
\tilde{C}^r(\R^d) \subset C^r(\R^d).
$$

\begin{theorem}
\label{sobolev embedding}(Sobolev Embedding) Suppose that $s>d/2$ is a real number. For any $j \geq 0$, 
$$H_{ul}^{j+s}(\R^d) \hookrightarrow \tilde{C}^j(\R^d).$$ 
\end{theorem}
\begin{proof}
Let $f \in H_{ul}^{j+s}(\R^d)$, then by the classical Sobolev embedding theorem, 
$$
\sup_{x \in \mathbb{R}^d}\norm{\phi_xf}{\tilde{C}^j} \lesssim \sup_{x \in \R^d}\norm{\phi_xf}{H^{j+s}} = \norm{f}{H^{j+s}_{ul}}.
$$
If $j$ is not an integer, observe that 
\begin{align*}
\norm{f}{\tilde{C}^j}&=\sum_{|\alpha|\leq\floor{j}}\norm{D^{\alpha}f}{L^{\infty}}+\sum_{|\beta|=\floor{j}}\sup_{x \neq y} \frac{\abs{D^{\beta}f(x)-D^{\beta}f(y)}}{\abs{x-y}^{j-\floor{j}}}\\
&\leq 
3\sum_{|\alpha|\leq\floor{j}}\norm{D^{\alpha}f}{L^{\infty}}+\sum_{|\beta|=\floor{j}}\sup_{x \in \R^d}\left\{\sup_{y \in B_1(x)} \frac{\abs{D^{\beta}f(x)-D^{\beta}f(y)}}{\abs{x-y}^{j-\floor{j}}}\right\}\\
&\lesssim 
 \sup_{x \in \R^d}\sum_{|\alpha|\leq\floor{j}}\norm{D^{\alpha}(\phi_{x,2}f)}{L^{\infty}}+\sum_{|\beta|=\floor{j}}\sup_{x\in\R^d}\left\{\sup_{y \in B_1(x)} \frac{\abs{D^{\beta}(\phi_{x,2}f)(x)-D^{\beta}(\phi_{x,2}f)(y)}}{\abs{x-y}^{j-\floor{j}}}\right\}\\
&\lesssim  \sup_{x \in \R^d}\sum_{|\alpha|\leq\floor{j}}\norm{D^{\alpha}(\phi_{x,2}f)}{L^{\infty}}+\sum_{|\beta|=\floor{j}}\sup_{x\in\R^d}\left\{\sup_{x\neq y} \frac{\abs{D^{\beta}(\phi_{x,2}f)(x)-D^{\beta}(\phi_{x,2}f)(y)}}{\abs{x-y}^{j-\floor{j}}}\right\}\\
&=\sup_{x \in \R^d}\norm{\phi_{x,2}f}{\tilde{C}^j}\lesssim \sup_{x \in \R^d} \norm{\phi_xf}{\tilde{C}^j}.
\end{align*}
In the first inequality we used that 
\begin{align*}
\sup_{x \neq y} \frac{\abs{D^{\beta}f(x)-D^{\beta}f(y)}}{\abs{x-y}^{j-\floor{j}}} &\leq 
\sup_{y \in B_1(x)} \frac{\abs{D^{\beta}f(x)-D^{\beta}f(y)}}{\abs{x-y}^{j-\floor{j}}} + \sup_{\abs{x-y} \geq 1} \frac{\abs{D^{\beta}f(x)-D^{\beta}f(y)}}{\abs{x-y}^{j-\floor{j}}}\\ & \leq \sup_{y \in B_1(x)} \frac{\abs{D^{\beta}f(x)-D^{\beta}f(y)}}{\abs{x-y}^{j-\floor{j}}} + 2\norm{D^{\beta}f}{L^{\infty}}. 
\end{align*}
If $j$ is an integer we repeat the previous argument without the semi-norm estimate.
Thus we have shown $\norm{f}{\tilde{C^j}} \lesssim  \norm{f}{H_{ul}^s}$ which concludes the proof.
\end{proof}

\section{Generalized SQG in Uniformly Local Sobolev Spaces}
We aim to show well-posedness of \eqref{gSQG} in $H^s_{ul}(\R^2)$. Our strategy is to create a sequence of approximate solutions as described in the proof of Theorem \ref{Holder Result} and show that the $H_{ul}^s(\R^2)$ norm of the sequence is uniformly bounded. This uniform bound will allow us to pass to the limit and obtain a solution in $H^s_{ul}(\R^2)$. We first need an a-priori estimate on the $H_{ul}^s(\R^2$) norm of smooth solutions to \eqref{gSQG}. Before proving this estimate, we establish several useful lemmas.

\begin{lemma}
\label{Bernstein for Riesz Potentital}
Suppose that $\theta \in \mathcal{S}'(\R^d)$ and $1\leq p \leq \infty$. Then for any $j \in \mathbb{Z}$ and $s>0$ we have 
\begin{equation} 
\norm{\dot{\Delta}_j (-\Delta)^{\frac{s}{2}}\theta}{L^p} \sim 2^{js}\norm{\dot{\Delta}_j \theta}{L^p}.
\end{equation}
\end{lemma}
\begin{proof}
See (A.4) in Appendix A from \cite{13}.
\end{proof}

\begin{lemma}
\label{near field Hsul estimate}
Let $\beta \in (0,1)$ and $s>1-\beta$.Suppose $\theta \in L_{ul}^2 \cap \dot{H}_{ul}^{s-1+\beta}(\R^2)$, then 
\begin{equation}
\label{near field homogeneuous velocity estimate}
\norm{\nabla^{\perp}(a \Phi_{\beta})\ast \theta}{\dot{H}_{ul}^s} \lesssim \norm{\theta}{\dot{H}_{ul}^{s-1+\beta}} + \norm{\theta}{L_{ul}^2}.
\end{equation}
\end{lemma}
\begin{proof}
We first note that for any $z\in \R^2$  
\begin{equation}
\phi_z(\nabla^{\perp}(a\Phi_{\beta})\ast\theta)=\phi_z(\nabla^{\perp}(a\Phi_{\beta})\ast(\phi_{z,8}\theta)).
\end{equation}
Then applying the Leibniz rule and Hölder's inequality gives
\begin{align*}
\norm{\nabla^{\perp}(a \Phi_{\beta})\ast \theta}{\dot{H}_{ul}^s} &= \norm{(-\Delta)^{\frac{s}{2}}(\phi_z(\nabla^{\perp}(a\Phi_{\beta})\ast(\phi_{z,8}\theta)))}{L^2}\\
&\leq \norm{(-\Delta)^{\frac{s}{2}}\phi_z}{L^{\infty}}\norm{\nabla^{\perp}(a\Phi_{\beta})\ast(\phi_{z,8}\theta)}{L^2}\\ & \hspace{1cm}+ \norm{\phi_z}{L^{\infty}}\norm{(-\Delta)^{\frac{s}{2}}(\nabla^{\perp}(a\Phi_{\beta}) \ast (\phi_{z,8}\theta))}{L^2}\\
&\lesssim \norm{\theta}{L_{ul}^2} + \norm{(-\Delta)^{\frac{s}{2}}(\nabla^{\perp}(a\Phi_{\beta}) \ast (\phi_{z,8}\theta))}{L^2},
\end{align*}
where in the last inequality we used Young's inequality and the fact that $\nabla^{\perp}(a\Phi_{\beta})\in L^1(\R^2)$. Thus, it suffices to show 
$$\norm{\nabla^{\perp}(a\Phi_{\beta}) \ast \theta}{\dot{H}^s}  \lesssim \norm{\theta}{\dot{H}^{s-1+\beta}}. 
$$ 
Using a Littlewood-Paley Decomposition we write
\begin{equation}
\label{LPD pre bernstein}
\norm{\nabla^{\perp}(a\Phi_{\beta})\ast \theta}{\dot{H}^s}^2 = \sum_{j \in \mathbb{Z}} 2^{2js}\norm{\dot{
\Delta}_j \nabla^{\perp}(a\Phi_{\beta})\ast \theta}{L^2}^2.
\end{equation}
Since $\mathcal{F}\left(\dot{
\Delta}_j \nabla^{\perp}(a\Phi_{\beta})\ast \theta\right)(\xi)$ is supported away from the origin we note 
\begin{align*}
\norm{\dot{
\Delta}_j \nabla^{\perp}(a\Phi_{\beta})\ast \theta}{L^2} &=  \norm{(-\Delta)^{1-\frac{\beta}{2}}(a\Phi_{\beta})\ast(\dot{\Delta}_j(-\Delta)^{\frac{\beta}{2}-1}\nabla^{\perp}\theta)}{L^2} \\
&=
\norm{\mathcal{F}\left((-\Delta)^{1-\frac{\beta}{2}}(a\Phi_{\beta})\right)(\xi)\mathcal{F}\left(\dot{\Delta}_j(-\Delta)^{\frac{\beta}{2}-1}\nabla^{\perp}\theta\right)(\xi)}{L_{\xi}^2} \\
&\leq
\norm{\mathcal{F}\left((-\Delta)^{1-\frac{\beta}{2}}(a\Phi_{\beta})\right)}{L_{\xi}^{\infty}}\norm{\dot{\Delta}_j(-\Delta)^{\frac{\beta}{2}-1}\nabla^{\perp}\theta}{L^2} \\
&\lesssim 2^{j(-1+\beta)}\norm{\mathcal{F}\left((-\Delta)^{1-\frac{\beta}{2}}(a\Phi_{\beta})\right)}{L_{\xi}^{\infty}}\norm{\dot{\Delta}_j\theta}{L^2}.
\end{align*}
Substituting the last inequality into \eqref{LPD pre bernstein} we see that 
$$
\norm{\nabla^{\perp}(a\Phi_{\beta})\ast\theta}{\dot{H}^s} \leq \norm{\mathcal{F}\left((-\Delta)^{1-\frac{\beta}{2}}(a\Phi_{\beta})\right)}{L_{\xi}^{\infty}} \norm{\theta}{\dot{H}^{s-1+\beta}}.
$$
It remains to show $\mathcal{F}\left((-\Delta)^{1-\frac{\beta}{2}}(a\Phi_{\beta})\right) \in L_{\xi}^{\infty}(\R^2)$. Fix $\xi \in \R^2$, then 
\begin{align*}
\left|\mathcal{F}\left((-\Delta)^{1-\frac{\beta}{2}}(a\Phi_{\beta})\right)(\xi)\right|&= \left| C\abs{\xi}^{2-\beta}\int_{\R^2}\hat{a}(\eta-\xi)\abs{\eta}^{\beta-2}d\eta \right| \\
&\leq \int_{\R^2} \abs{\hat{a}(\eta-\xi)}\frac{\abs{(\xi-\eta)+\eta}^{2-\beta}}{\abs{\eta}^{2-\beta}}d\eta \\
&\lesssim_{\beta} \int_{\R^2}\abs{\hat{a}(\eta-\xi)}\frac{\abs{\eta-\xi}^{2-\beta}+\abs{\eta}^{2-\beta}}{\abs{\eta}^{2-\beta}}d\eta\\
&\lesssim 
\norm{\hat{a}}{L_{\xi}^1} + \int_{\R^2}\frac{\abs{\hat{a}(\eta-\xi)}\abs{\eta-\xi}^{2-\beta}}{\abs{\eta}^{2-\beta}}d\eta.
\end{align*}
In the second inequality we used Jensen's inequality to see $\abs{x+y}^r \leq 2^{r-1}\left(\abs{x}^r+\abs{y}^r\right)$ for $r>1$, and in the third inequality we utilized the  translation invariance of the $L^1$ norm. Continuing, we see that 
\begin{align*}
\int_{\R^2}\frac{\abs{\hat{a}(\eta-\xi)}\abs{\eta-\xi}^{2-\beta}}{\abs{\eta}^{2-\beta}}d\eta &\leq
\int_{|\eta|< 1}\frac{\abs{\hat{a}(\eta-\xi)}\abs{\eta-\xi}^{2-\beta}}{\abs{\eta}^{2-\beta}}d\eta\\ &\hspace{1cm} + \int_{|\eta| \geq 1}\frac{\abs{\hat{a}(\eta-\xi)}\abs{\eta-\xi}^{2-\beta}}{\abs{\eta}^{2-\beta}}d\eta\\
&\leq
\norm{\hat{a}(\eta-\xi)\abs{\eta-\xi}^{2-\beta}}{L_{\eta}^{\infty}}\norm{\abs{\eta}^{-2+\beta}}{L_{\eta}^1(B_1(0))} \\ &\hspace{.5cm}+ \norm{\abs{\eta}^{-2+\beta}}{L_{\eta}^{\infty}(\R^2\setminus B_1(0))}\norm{\hat{a}(\eta-\xi){\abs{\eta-\xi}^{2-\beta}}}{L_{\eta}^1}\\&\leq C<\infty.
\end{align*}
In the second inequality we used Hölder's inequality, while, to get the third inequality, we used the translation invariance of the $L^{\infty}$ and $L^1$ norm, and that $\hat{a}$ is a Schwartz function.
\end{proof}
A minor modification to the above argument yields the following result in non-homogeneous spaces.
\begin{lemma}
\label{near field nonhomogeneuous velocity estimate}
Let $\beta \in (0,1)$ and $s>1-\beta$. Suppose $\theta \in H^{s-1+\beta}_{ul}(\R^2)$, then 
$$
\norm{\nabla^{\perp}(a\Phi_{\beta})\ast \theta}{H^{s}_{ul}}\lesssim \norm{\theta}{H^{s-1+\beta}_{ul}}.
$$
\end{lemma}

We now prove an a-priori estimate on the $H^s_{ul}$ norm of smooth solutions to \eqref{gSQG}.
\begin{theorem}
\label{Hsul apriori}
Suppose that $s > 2$ and $(u,\theta)$ is a smooth solution to \eqref{gSQG} on $[0,T]$ as in Theorem \ref{Holder Result}, with Hölder exponent $r=s+2$. Then for any $t \in [0,T]$,
$$
\norm{\theta(t)}{H_{ul}^s}\leq \norm{\theta^0}{H_{ul}^s}\exp{\left(\int_0^t \left(\norm{u(\tau)}{\tilde{C}^1} + \norm{\theta(\tau)}{W^{1,\infty}}\right)d\tau\right)}.
$$
\end{theorem}
\begin{proof}
We multiply \eqref{gSQG}$_1$ by $\phi_z$, apply the operator $J^s$, and take an inner product against $J^s(\phi_z \theta)$ to see that 
\begin{equation}
\label{eq0}
\frac{1}{2}\frac{d}{dt}\norm{J^s(\phi_z\theta)}{L^2}^2=- (J^s(\phi_z(u\cdot\nabla\theta)),J^s(\phi_z\theta))_{L^2}.
\end{equation}
Due to the support of $\phi_z$ we see that $J^s(\phi_z(u\cdot\nabla\theta))=J^s(\phi_z(u\cdot\nabla(\phi_{z,2}\theta)))$. We then add $(\phi_zu\cdot J^s(\nabla(\phi_{z,2}\theta)),J^s(\phi_z\theta))_{L^2}$ to both sides of \eqref{eq0} to see that 
\begin{equation}
\label{earliest apriori hsul eqn}
\frac{1}{2}\frac{d}{dt}\norm{J^s(\phi_z\theta)}{L^2}^2 + (\phi_z u \cdot \nabla J^s(\phi_{z,2}\theta),J^s(\phi_z \theta))_{L^2} = (F,J^s(\phi_z\theta))_{L^2},
\end{equation}
where $F= \phi_zu\cdot J^s(\nabla(\phi_{z,2}\theta))-J^s(\phi_zu\cdot\nabla(\phi_{z,2}\theta))$. By the Cauchy-Schwarz inequality and Proposition \ref{bessel commutator}
\begin{align*}
(F,J^s(\phi_z\theta))_{L^2}&\leq
\norm{J^s(\phi_z\theta)}{L^2}\norm{\phi_z u\cdot J^s(\nabla(\phi_{z,2}\theta))-J^s(\phi_zu\cdot(\nabla(\phi_{z,2}\theta)))}{L^2}\\
&\leq 
\sup_{z \in \R^2} \norm{J^s(\phi_z\theta)}{L^2} \norm{\nabla(\phi_zu)}{L^{\infty}}\norm{J^{s-1}(\nabla(\phi_{z,2}\theta))}{L^2}\\ &\hspace{.5cm}+\sup_{z \in \R^2} \norm{J^s(\phi_z\theta)}{L^2}\norm{J^s(\phi_z u)}{L^2}\norm{\nabla(\phi_{z,2}\theta)}{L^{\infty}}\\
&\leq
\norm{\theta}{H^s_{ul}}^2\norm{u}{\tilde{C}^1}+\norm{\theta}{H_{ul}^s}\norm{u}{H^s_{ul}}\norm{\theta}{\tilde{C}^1}.
\end{align*}
We next estimate $(\phi_z u \cdot \nabla J^s(\phi_{z,2}\theta),J^s(\phi_z \theta))_{L^2}$ in \eqref{earliest apriori hsul eqn}. Note that by the product rule,
\begin{equation}
\label{reverse product rule for Hsul apriori}
\phi_z u \cdot \nabla J^s(\phi_{z,2}\theta) = u\cdot\nabla(\phi_zJ^s(\phi_{z,2}\theta))-(u\cdot\nabla\phi_z)J^s(\phi_{z,2}\theta).
\end{equation}
The first term on the right hand side of \eqref{reverse product rule for Hsul apriori} does not vanish when integrated against $J^s(\phi_{z}\theta)$. Instead, we utilize the support of $\phi_z$ and introduce another commutator as follows:
\begin{align*}
&(u\cdot\nabla(\phi_zJ^s(\phi_{z,2}\theta)),J^s(\phi_z\theta))_{L^2} = \int_{\R^2} u\cdot\nabla(\phi_zJ^s(\phi_{z,2}\theta)) J^s(\phi_z\theta) dx\\
&=\int_{\R^2} u\cdot\nabla(\phi_zJ^s(\phi_{z,2}\theta))(J^s(\phi_z\phi_{z,2}\theta)-\phi_zJ^s(\phi_{z,2}\theta)+\phi_zJ^s(\phi_{z,2}\theta))dx\\
&=
\int_{\R^2}u\cdot\nabla(\phi_zJ^s(\phi_{z,2}\theta))\phi_zJ^s(\phi_{z,2}\theta) dx + \int_{\R^2}u\cdot\nabla(\phi_zJ^s(\phi_{z,2}\theta))[J^s,\phi_z](\phi_{z,2}\theta) dx\\
&= I_1 + I_2,
\end{align*}
where we define
\begin{align*}
I_1 & = \int_{\R^2}u\cdot\nabla(\phi_zJ^s(\phi_{z,2}\theta))\phi_zJ^s(\phi_{z,2}\theta) dx,\\
I_2 &= \int_{\R^2}u\cdot\nabla(\phi_zJ^s(\phi_{z,2}\theta))[J^s,\phi_z](\phi_{z,2}\theta) dx.
\end{align*}
By the chain rule and integration by parts, 
$$
I_1 = -\frac{1}{2} \int_{\R^2} (\text{div }u)(\phi_zJ^s(\phi_{z,2}\theta))^2 dx = 0.
$$
To estimate $I_2$, we note that 
$$
u\cdot\nabla(\phi_zJ^s(\phi_{z,2}\theta)) = u\phi_{z,2}\cdot\nabla(\phi_zJ^s(\phi_{z,2}\theta)),
$$
integrate by parts and use that $u$ is divergence free. This gives 
\begin{align*}
I_2 &= -\int_{\R^2} \phi_z J^s(\phi_{z,2}\theta)(u\cdot\nabla \phi_{z,2})[J^s,\phi_z](\phi_{z,2}\theta) dx\\ & \hspace{.5cm}- \int_{\R^2} \phi_z J^s(\phi_{z,2}\theta)u\phi_{z,2}\cdot \nabla \left([J^s,\phi_z](\phi_{z,2}\theta)\right)dx \\
&= -(I_{21}+I_{22}),
\end{align*}
where we define 
\begin{align*}
I_{21} &= \int_{\R^2} \phi_z J^s(\phi_{z,2}\theta)(u\cdot\nabla \phi_{z,2})[J^s,\phi_z](\phi_{z,2}\theta) dx \\
I_{22} &= \int_{\R^2} \phi_z J^s(\phi_{z,2}\theta)u\phi_{z,2}\cdot \nabla \left([J^s,\phi_z](\phi_{z,2}\theta)\right)dx.
\end{align*}
Applying the Cauchy-Schwarz inequality, the Sobolev embedding theorem,  and Proposition \ref{bessel commutator} we find that 
\begin{align*}
I_{21} &\leq  \sup_{z \in \R^2} \norm{\phi_zJ^s(\phi_{z,2}\theta)u\cdot \nabla \phi_{z,2}}{L^2}\norm{[J^s,\phi_z](\phi_{z,2}\theta)}{L^2}\\
&\lesssim \norm{u}{L^{\infty}}\norm{\theta}{H^s_{ul}}\sup_{z \in \R^2}\left(\norm{J^{s-1}(\phi_{z,2}\theta)}{L^2}\norm{\nabla\phi_z}{L^{\infty}}+\norm{J^s\phi_z}{L^2}\norm{\phi_{z,2}\theta}{L^{\infty}}\right)\\
&\lesssim
\norm{\theta}{H_{ul}^s}^2 \norm{u}{L^{\infty}}.
\end{align*}
The term $I_{22}$ is more involved. Again by the Cauchy-Schwarz inequality, 
\begin{align*}
I_{22} &\leq \norm{\theta}{H_{ul}^s}\norm{u}{L^{\infty}}\norm{\nabla\left(J^s(\phi_z\phi_{z,2}\theta)-\phi_{z}J^s(\phi_{z,2}\theta)\right)}{L^2}.
\end{align*}
The product rule and Proposition \ref{bessel commutator} give 
\begin{align*}
\norm{\nabla([J^s,\phi_z]\phi_{z,2}\theta)}{L^2} 
&\leq \norm{J^s((\nabla\phi_z)\phi_{z,2}\theta) - (\nabla \phi_z)J^s(\phi_{z,2}\theta)}{L^2}\\
& \hspace{.5cm}+ \norm{J^s(\phi_z\nabla(\phi_{z,2}\theta)-\phi_zJ^s(\nabla(\phi_{z,2}\theta)}{L^2}
\\
&\lesssim \norm{\phi_z}{\tilde{C}^2}\norm{J^{s-1}(\phi_{z,2}\theta)}{L^2}+\norm{\phi_{z,2}\theta}{L^{\infty}}\norm{J^{s}(\nabla\phi_z)}{L^2} \\
&  \hspace{.5cm}+ \norm{J^{s-1}(\nabla(\phi_{z,2}\theta))}{L^2}\norm{\nabla \phi_z}{L^{\infty}}+\norm{J^s\phi_z}{L^2}\norm{\nabla(\phi_{z,2}\theta)}{L^{\infty}}\\
&\lesssim \norm{\theta}{H_{ul}^s} + \norm{\theta}{\tilde{C}^1}.
\end{align*}
Applying Theorem \ref{sobolev embedding} we find that $$I_{22} \lesssim \norm{\theta}{H_{ul}^s}^2\norm{u}{L^{\infty}}.$$
Hence,
\begin{equation}
\label{pre gronwall and integration Hsul}
\frac{1}{2}\frac{d}{dt}\norm{\theta}{H_{ul}^s}^2 \lesssim \norm{\theta}{H_{ul}^s}^2\norm{u}{\tilde{C}^1} + \norm{\theta}{H_{ul}^s}\norm{u}{H_{ul}^s}\norm{\theta}{W^{1,\infty}}.
\end{equation}
Then integrating \eqref{pre gronwall and integration Hsul} in time yields
\begin{equation}
\label{a-priori pre velocity estimate Hsul}
\norm{\theta(t)}{H_{ul}^s}^2 \lesssim \norm{\theta_0}{H_{ul}^s}^2 + \int_0^t \left (\norm{\theta(\tau)}{H_{ul}^s}^2\norm{u(\tau)}{\tilde{C}^1} + \norm{\theta(\tau)}{H_{ul}^s}\norm{u(\tau)}{H_{ul}^s}\norm{\theta(\tau)}{W^{1,\infty}} \right) d\tau.
\end{equation}
It remains to estimate $\norm{u(\tau)}{H_{ul}^s}$ and then apply Grönwall's lemma to close the estimate. Recall that $(u,\theta)$ satisfy \eqref{frequency equality in Holder space}. In particular we have
$$
\dot{\Delta}_j u = \dot{\Delta}_j\nabla^{\perp}(-\Delta)^{-1+\frac{\beta}{2}}\theta \hspace{1cm} \text{ for all } j \in \mathbb{Z}.
$$ 
So for $r' \in (0,s]$ we have that 
$$
\dot{\Delta}_j(-\Delta)^{\frac{r'}{2}}u = \dot{\Delta}_j (-\Delta)^{-1+\frac{\beta+r'}{2}}\theta.
$$
Note that if $r'>1-\beta$ then $(-\Delta)^{\frac{r'}{2}}(\nabla^{\perp}((1-a)\Phi_{\beta})) \in L^1(\R^2))$. We fix $r' \in (1-\beta,s]$, then there exists some polynomial $P$ such that 
\begin{equation}
\label{eq1} (-\Delta)^{\frac{r'}{2}}u = P + \nabla^{\perp}(a\Phi_{\beta})\ast (-\Delta)^{\frac{r'}{2}}\theta + (-\Delta)^{\frac{r'}{2}}(\nabla^{\perp}((1-a)\Phi_{\beta}))\ast \theta.
\end{equation}
We claim that $(-\Delta)^{\frac{r'}{2}}u \in L^{\infty}$. Indeed since $u \in C^{r+1-\beta}$ one has 
\begin{align*}
\norm{(-\Delta)^{\frac{r'}{2}}u}{L^{\infty}}&\leq \sum_{j \geq -1} \norm{\Delta_j(-\Delta)^{\frac{r'}{2}}u}{L^{\infty}}\\
&\leq \sum_{j \geq -1} 2^{j(r'-r-1+\beta)}2^{j(r+1-\beta)}\norm{\Delta_ju}{L^{\infty}}\\
&\leq \sup_{j \geq -1} \left\{ 2^{j(r+1-\beta)}\norm{\Delta_ju}{L^{\infty}}\right\} \sum_{j \geq -1} 2^{j(r'-r-1+\beta)}\\
&\lesssim \norm{u}{C^{r+1-\beta}}.
\end{align*}
In the last inequality we used that $r'<r$ and $-1+\beta<0$, so $r'-r-1+\beta<0$. An identical argument shows $(-\Delta)^{\frac{r'}{2}}\theta \in L^{\infty}$. Applying Young's convolution inequality, we conclude that the polynomial $P$ is constant. So, for any $r \in (r',s]$, 
\begin{equation}
\label{hg estimate apriori}
(-\Delta)^{\frac{r}{2}} u = \nabla^{\perp}(a\Phi_{\beta}) \ast (-\Delta)^{\frac{r}{2}}\theta + (-\Delta)^{\frac{r}{2}}(\nabla^{\perp}((1-a)\Phi_{\beta})) \ast \theta,
\end{equation}
where above we used that the fractional Laplacian of any constant in zero.
Upon taking the $\dot{H}_{ul}^{s-r}$ norm of both sides and applying Lemma \ref{near field Hsul estimate} we find that 
$$
\norm{u}{\dot{H}_{ul}^{s}} \lesssim \norm{\theta}{\dot{H}_{ul}^{s-1+\beta}}+\norm{\theta}{L^2_{ul}}+\norm{\theta}{L^{\infty}} \lesssim \norm{\theta}{H_{ul}^{s-1+\beta}}.
$$
where in the first inequality we used the embedding $L^{\infty}(\R^2) \hookrightarrow L^2_{ul}(\R^2)$ and Young's inequality to conclude
$$
\norm{\nabla^{\perp}((1-a)\Phi_{\beta}) \ast \theta}{\dot{H}^s_{ul}}\leq \norm{\phi}{L^{\infty}}\norm{(-\Delta)^{s/2}(\nabla^{\perp}((1-a)\Phi_{\beta}))}{L^1}\norm{\theta}{L^{\infty}},
$$
and in the second we used the Sobolev embedding theorem. Applying Proposition \ref{SS equivalent norm} we find
$$
\norm{u}{H_{ul}^s} \sim \norm{u}{W_{ul}^{\floor{s},2}} + \norm{u}{\dot{H}_{ul}^s} \lesssim \norm{\theta}{H_{ul}^{s-1+\beta}}+\sum_{\abs{\alpha}=0}^{\floor{s}} \norm{D^{\alpha} u}{L_{ul}^2}.
$$
Applying an identical argument that was used to obtain \eqref{eq1}, we can also show that there exists some polynomial $\tilde{P}$ such that 
\begin{equation}
\label{eq2}
D^{\alpha} u = \tilde{P} + \nabla^{\perp}(a\Phi_{\beta}) \ast D^{\alpha} \theta + D^{\alpha} \nabla^{\perp}((1-a)\Phi_{\beta}) \ast \theta,
\end{equation}
where $D^{\alpha}$ is a differential operator with $1 \leq \abs{\alpha}\leq \floor{s}$. Since $u,\theta \in \tilde{C}^1$ we see that the polynomial $\tilde{P}$ is a constant, so for $2 \leq \abs{\alpha} \leq \floor{s}$,  
\begin{equation}
\label{eq3}
D^{\alpha} u = \nabla^{\perp}(a\Phi_{\beta}) \ast D^{\alpha} \theta + D^{\alpha} \nabla^{\perp}((1-a)\Phi_{\beta}) \ast \theta.
\end{equation}
Summing \eqref{eq3} for $\abs{\alpha}=2,\cdots,\floor{s}$  and applying Lemma \ref{near field nonhomogeneuous velocity estimate} yields
$$
\sum_{\abs{\alpha}=2}^{\floor{s}} \norm{D^{\alpha} u}{L_{ul}^2} \lesssim \norm{\theta}{H_{ul}^{s-1+\beta}}+\norm{\theta}{L^{\infty}}\lesssim \norm{\theta}{H^{s-1+\beta}_{ul}}.
$$
Estimating the lower order derivatives by $\norm{u}{L_{ul}^2} + \norm{\nabla u}{L_{ul}^2} \lesssim  \norm{u}{\tilde{C}^1}$ gives
 \begin{equation}
 \label{velocitu hsul estimate}
 \norm{u}{H^s_{ul}} \lesssim \norm{\theta}{H_{ul}^{s-1+\beta}}+\norm{u}{\tilde{C}^1}.
 \end{equation}
 Then substituting \eqref{velocitu hsul estimate} into \eqref{a-priori pre velocity estimate Hsul} gives
\begin{align*}
\norm{\theta(t)}{H_{ul}^s}^2 & \lesssim \norm{\theta^0}{H_{ul}^s}^2 + \int_0^t \norm{\theta(\tau)}{H_{ul}^s}^2(\norm{u(\tau)}{\tilde{C}^1} + \norm{\theta(\tau)}{W^{1,\infty}})d\tau \\
&+\int_0^t \norm{\theta(\tau)}{H_{ul}^s}\norm{u(\tau)}{L^{\infty}}\norm{\theta(\tau)}{W^{1,\infty}}d\tau.
\end{align*}
Using the Sobolev embedding theorem once again we find 
$$
\norm{\theta(t)}{H_{ul}^s}^2 \lesssim \norm{\theta^0}{H_{ul}^s}^2 + \int_0^t \norm{\theta(\tau)}{H_{ul}^s}^2(\norm{u(\tau)}{\tilde{C}^1} + \norm{\theta(\tau)}{W^{1,\infty}})d\tau.
$$
Applying Grönwall's Lemma then yields the desired result.
\end{proof} 
\begin{theorem}
\label{Hsul Result}
Let $s>2$ and $\beta \in (0,1)$. Assume $\theta^0 \in H_{ul}^s(\R^2)$ and $u^0\in H_{ul}^{s+1-\beta}(\R^2)$ satisfy
$$
u^0 = \nabla^{\perp}(-\Delta)^{-1+\beta/2}\theta^0 \hspace{.5cm} \text{ in } \dot{C}^{\alpha}(\R^2),
$$
for some $\alpha \in (1,s-1)$. Then there exists a $T>0$ and a unique solution $(u,\theta)$ to 
\begin{equation*}
\begin{cases}
    \partial_t \theta + u \cdot \nabla \theta=0\\
    (u,\theta)\vert_{t=0}=(u^0,\theta^0)
\end{cases}
\end{equation*}
satisfying
\begin{align*}
\theta &\in L^{\infty}([0,T];H_{ul}^s(\R^2)) \cap Lip([0,T];H_{ul}^{s-1}(\R^2)) \\
u &\in L^{\infty}([0,T]; H_{ul}^{s+1-\beta}(\R^2)). 
\end{align*}
Furthermore $(u,\theta)$ satisfies the Serfati-type identity
\begin{equation}
\label{Sobolev Serfati Identity}
u(t)=u^0 + \nabla^{\perp}(a\Phi_{\beta}) \ast (\theta(t)-\theta^0) - \int_0^t \nabla\nabla^{\perp}((1-a)\Phi_{\beta}) \ast\cdot (\theta u )(\tau)d\tau.
\end{equation}
\end{theorem}

\begin{proof}
\textbf{Approximating Sequence and Uniform Bounds}
We begin by taking the same approximating sequence as in the proof of Theorem \ref{Holder Result}. Then, since for each $n$, $u_n^0=S_{n 
 +1}u^0$ and $\theta_n^0=S_{n+1}\theta^0$, by Lemma \ref{Lp norm of Snf} and associativity of convolution one has that 
\begin{equation}
\label{Sobolev Uniform bound on initial H}
\norm{u_n^0}{H_{ul}^{s+1-\beta}} \lesssim \norm{u^0}{H_{ul}^{s+1-\beta}} \text{  and  } 
\norm{\theta_n^0}{H_{ul}^s} \lesssim \norm{\theta^0}{H_{ul}^s}.
\end{equation}
We first show there exists a single $T>0$ such that $(u_n)$ and $(\theta_n)$ are uniformly bounded in $L^{\infty}([0,T];H_{ul}^{s+1-\beta}(\R^2))$ and $L^{\infty}([0,T];H_{ul}^s(\R^2))$, respectively. Making use of \eqref{Sobolev Uniform bound on initial H} and the embedding $H_{ul}^{s}(\R^2) \hookrightarrow \tilde{C}^{\alpha}(\R^2)$ for $s>2$ and $\alpha \in (1,s-1)$, we see that $(u_n^0)$ is uniformly bounded in $C^{\alpha+1-\beta}(\R^2)$ and $(\theta_n^0)$ is uniformly bounded in $C^{\alpha}(\R^2)$. Thus by Theorem \ref{Holder Result}, $u_n,\theta_n$ exist in $C^{\alpha+1-\beta}(\R^2)$ and $C^{\alpha}(\R^2)$ respectively on $[0,T_n]$ with $(u_n,\theta_n)$ satisfying \eqref{Holder gSQG}, where (see \eqref{holder time})
\begin{equation}
\label{hsul time}
T_n < \frac{1}{C(\norm{u_n^0}{L^{\infty}} + \norm{\theta_n^0}{C^{\alpha}})}.
\end{equation}
By \eqref{Sobolev Uniform bound on initial H}, the above denominator is uniformly bounded in $n$. Thus, we choose a $T>0$ such that $T\leq T_n$ and
$$
\frac{1}{2C} \leq T_n(\norm{u_n^0}{L^{\infty}}+\norm{\theta_n^0}{C^{\alpha}})\leq T(\norm{u^0}{L^{\infty}}+\norm{\theta^0}{C^{\alpha}}) < \frac{1}{C},
$$
where $C$ is as in \eqref{hsul time}. Since $T\leq T_n$, by Theorem \ref{Holder Result} we have that for each $n$, $(u_n,\theta_n)$ is a solution to \eqref{gSQG} in the sense of Theorem \ref{Holder Result} on $[0,T]$, and by \eqref{bound on holder norms of solution}
$$
\norm{u_n}{C([0,T];L^{\infty})}+\norm{\theta_n}{C([0,T];C^{\alpha})} \lesssim \frac{\norm{u^0}{L^{\infty}}+\norm{\theta^0}{C^{\alpha}}}{1-CT(\norm{u^0}{L^{\infty}}+\norm{\theta^0}{C^{\alpha}})}.
$$
Then by Lemma \ref{Velocity Holder Estimate} $\norm{u_n}{C([0,T];C^{\alpha+1-\beta})}\lesssim \norm{u_n}{C([0,T];L^{\infty})}+\norm{\theta_n}{C([0,T];C^{\alpha})}$, so $(u_n),(\theta_n)$ are uniformly bounded in $C([0,T];C^{\alpha+1-\beta}(\R^2))$ and $C([0,T];C^{\alpha}(\R^2))$ respectively. Thus applying Theorem \ref{Hsul apriori} we find that there exists $C>0$ such that
$$
\norm{\theta_n}{C([0,T];H_{ul}^s)}\leq C.
$$
Using \eqref{velocitu hsul estimate}, we also see that there exists a constant $C_1>0$ satisfying
$$
\norm{u_n}{C([0,T];H_{ul}^{s+1-\beta})} \leq C_1.
$$
\\
\textbf{($\phi_R\theta_n$) is Cauchy.}
For ease of notation in what follows, we set $\phi_R=\phi_{0,R}$ where $\phi$ is the localization function in the definition of $H_{ul}^s(\R^2)$. Our next goal is to show that $(\phi_R\theta_n)$ is a Cauchy sequence in $C([0,T];H^{s-1}(\R^2))$ for every $R>0$. For each $n \in \mathbb{N}$
\begin{equation}
\label{Hsul eq1}
\partial_t\theta_n+u_n\cdot \nabla \theta_n=0,
\end{equation}
thus, multiplying \eqref{Hsul eq1} by $\phi_R$, taking the $H^{s-1}$ norm, and utilizing the uniform bounds on $\norm{u_n}{H_{ul}^{s-1}}$ and $\norm{\theta_n}{H_{ul}^s}$ yields
\begin{equation}
\label{Hsul eq2}
\norm{\phi_R\partial_t\theta_n}{H^{s-1}} = \norm{\phi_Ru_n\cdot \nabla \theta_n}{H^{s-1}} \leq C(R)\norm{u_n}{H_{ul}^{s-1}}\norm{\theta_n}{H_{ul}^s} \leq C(R).
\end{equation}
where we used that $H^{\gamma}_{ul}$ is a Banach algebra for $\gamma>1$.
For each $t \in [0,T]$ we have $\norm{\theta_n(t)}{H_{ul}^s}$ is uniformly bounded in $n$, so by the Rellich-Kondrachov Theorem, for each $R>0$ and $t>0$ there exists a subsequence of $(\phi_R\theta_n(t))$ which converges in $H^{s-1}(\R^2)$. A standard diagonalization argument then shows that for each $t>0$, there is a subsequence of $(\theta_n(t))$ (which we relabel to $(\theta_n(t))$) such that for every $R>0$ the sequence $(\phi_R\theta_n(t))$ converges in $H^{s-1}(\R^2)$.
\\
\\
We next aim to find a subsequence which converges for all $t\in [0,T]$. We will utilize the compactness of the time interval $[0,T]$ and \eqref{Hsul eq2}. By \eqref{Hsul eq2}, given $\eps>0$, there exists some $\delta>0$, such that whenever $\abs{t-s}<\delta$ we have 
\begin{equation}
\label{Hsul eq3}
\norm{\phi_R\theta_n(t)-\phi_R\theta_n(s)}{H^{s-1}}< \eps/3.
\end{equation}
We then partition our time interval $[0,T]$ into $0=t_0<t_1\hdots<t_N=T$ where $t_i-t_{i-1}<\delta$ for all $i=1,\hdots,N$. Note, since $\{t_i\}_{i=0}^n$ is a finite collection, we can find a further subsequence (which we relabel again) $(\phi_R\theta_n)$ that converges in $H^{s-1}(\R^2)$ for all $R>0$ and each $t_i$ in our partition. Thus there exists some $N \in \mathbb{N}$ such that whenever $n,m \geq N$ we have 
\begin{equation}
\label{Hsul eq4}
\norm{\phi_R\theta_n(t_i)-\phi_R\theta_m(t_i)}{H^{s-1}} < \eps/3,
\end{equation}
$\text{for each }i \text{ with } 0 \leq i \leq n.$
For $n,m \geq N$ and for each $t \in [0,T]$, we can choose $t_i$ satisfying $\abs{t_i-t}<\delta$ and apply \eqref{Hsul eq3} and \eqref{Hsul eq4} to conclude that
\begin{multline}
\norm{\phi_R\theta_n(t)-\phi_R\theta_m(t)}{H^{s-1}} \leq \norm{\phi_R\theta_n(t)-\phi_R\theta_n(t_i)}{H^{s-1}}\\
+ \norm{\phi_R\theta_n(t_i)-\phi_R\theta_m(t_i)}{H^{s-1}}+\norm{\phi_R\theta_m(t_i)-\phi_R\theta_m(t)}{H^{s-1}}<\eps.
\end{multline}
Thus $(\phi_R\theta_n)$ is a Cauchy sequence in $C([0,T];H^{s-1}(\R^2))$, so there exists some $\theta$ such that $\phi_R\theta_n \to \phi_{2R}\theta$  in $C([0,T];H^{s-1}(\R^2))$ for all $R>0$. To see why the limit must be of the form $\phi_{2R}\theta$, note that $H^{s-1}\hookrightarrow L^{\infty}$ so the convergence is uniform. Thus, the limit of $(\phi_R\theta_n)$ cannot have larger support than $(\phi_R\theta_n)$.
\\
\\
\textbf{$(\phi_Ru_n)$ is Cauchy.}
We next aim to show that $(\phi_Ru_n)$ is Cauchy in $C([0,T];H^{s-\beta}(\R^2))$. Utilizing the uniform bound on $\norm{u_n(t)}{H_{ul}^{s+1-\beta}}$, the Rellich-Kondrachov theorem, and a standard diagonalization argument we can find some subsequence of $(\phi_Ru_n(t))$ that converges in $H^{s-\beta}(\R^2)$ for every $R>0$. It remains to show that we can find a subsequence that works for all $t\in [0,T]$. To do this, we will make use of the Serfati-type identity for $u_n$. Observe that
\begin{equation}
\label{Hsul eqn5}
u_n(t)-u_n(s) = \nabla^{\perp}(a\Phi_{\beta}) \ast (\theta_n(t)-\theta_n(s)) - \int_s^t (\nabla\nabla^{\perp}((1-a)\Phi_{\beta})\ast\cdot(\theta_nu_n))(\tau)d\tau.
\end{equation}
Taking the $H^{s-\beta}$ norm of \eqref{Hsul eqn5} and applying Lemma \ref{near field nonhomogeneuous velocity estimate} we find that for each $R>0$
\begin{multline}
\label{Hsul eqn6}
\norm{\phi_Ru_n(t)-\phi_Ru_n(s)}{H^{s-\beta}} \leq \norm{\phi_R\left(\nabla^{\perp}( a\Phi_{\beta})\ast(\phi_{8R}\theta_n(t)-\phi_{8R}\theta_n(s))\right)}{H^{s-\beta}}\\
+
\int_s^t \norm{J^{s-\beta}\phi_R\nabla\nabla^{\perp}((1-a)\Phi_{\beta})\ast \cdot (\theta_nu_n)(\tau)}{L^2} d\tau
\\
\lesssim_{R,\beta} \norm{\phi_{8R}\theta_n(t)-\phi_{8R}\theta_n(s)}{H^{s-1}}  
+ \int_s^t \norm{J^{s-\beta}\phi_R\nabla\nabla^{\perp}((1-a)\Phi_{\beta})\ast \cdot (\theta_nu_n)(\tau)}{L^2} d\tau.
\end{multline}
Handling the last integral requires care. By the Leibniz rule note that
\begin{align*}
\norm{J^{s-\beta}(\phi_R\nabla\nabla^{\perp}((1-a)(\Phi_{\beta})\ast \cdot (\theta_nu_n))}{L^2}& \lesssim \norm{J^{s-\beta}\phi_R}{L^2}\norm{\nabla\nabla^{\perp}((1-a)\Phi_{\beta})\ast \cdot (\theta_nu_n)}{L^{\infty}}\\ & \hspace{-4.5cm}+\norm{J^{s-\beta}\nabla\nabla^{\perp}((1-a)\Phi_{\beta})\ast \cdot (\theta_nu_n)}{L^{\infty}}\norm{\phi_R}{L^2}\\ &\hspace{-5cm}\lesssim_R \norm{\nabla\nabla^{\perp}((1-a)\Phi_{\beta})\ast \cdot (\theta_nu_n)}{L^{\infty}}
+\norm{J^{s-\beta}\nabla\nabla^{\perp}((1-a)\Phi_{\beta})\ast \cdot (\theta_nu_n)}{L^{\infty}}
\\
&\hspace{-5cm}\lesssim_{R,\beta} \norm{\theta_n}{L^{\infty}}\norm{u_n}{L^{\infty}}.
\end{align*}
Continuing the estimate from \eqref{Hsul eqn6}, we find that 
\begin{align*}
\norm{\phi_Ru_n(t)-\phi_Ru_n(s)}{H^{s-\beta}}&\lesssim_{R,\beta}  \norm{\phi_{8R}\theta_n(t)-\phi_{8R}\theta_n(s)}{H^{s-1}}\\& \hspace{.5cm}+\sup_{\tau \in [s,t]}\left\{ \norm{\theta_n(\tau)}{L^{\infty}}\norm{u_n(\tau)}{L^{\infty}}\right\}|t-s|
\\ &\lesssim_R |t-s| + \sup_{\tau \in [s,t]}\left\{ \norm{\theta_n(\tau)}{L^{\infty}}\norm{u_n(\tau)}{L^{\infty}}\right\}|t-s|,
\end{align*}
where in the last inequality we used \eqref{Hsul eq2}.
Therefore, given some $\eps>0$, there exists some $\delta>0$, such that whenever $\abs{t-s}<\delta$ we have
\begin{equation}
\norm{\phi_Ru_n(t)-\phi_Ru_n(s)}{H^{s-\beta}} < \eps.
\end{equation}
We take advantage of the compactness of the time interval $[0,T]$ and repeat the same argument we made to show that $(\phi_R\theta_n)$ is Cauchy in $C([0,T];H^{s-1}(\R^2))$. We conclude that $(\phi_Ru_n)$ is Cauchy in $C([0,T];H^{s-\beta}(\R^2))$, thus there exists some $u$ such that $\phi_Ru_n \to \phi_{2R} u$ in  $C([0,T];H^{s-\beta}(\R^2))$ for all $R>0$.
\\
\\
\textbf{$(u,\theta)$ satisfies Theorem \ref{Hsul Result} }
We next show that $(u,\theta)$ satisfies Theorem \ref{Hsul Result}. We first note that given $R>0$, for $n,m\in \mathbb{N}$ one has 
\begin{multline}
\label{Hsul eqn7}
\phi_R\partial_t(\theta_n-\theta_m) = \phi_R(u_n-u_m)\cdot \nabla \theta_m + \phi_Ru_n\cdot \nabla (\theta_n-\theta_m)\\
= \phi_R(u_n-u_m)\cdot(\phi_{2R}\nabla \theta_m) + \phi_Ru_n\cdot \phi_{2R}\nabla(\theta_n-\theta_m).
\end{multline}
Thus, taking the $H^{s-2}$ norm of \eqref{Hsul eqn7} yields
\begin{align*}
\norm{\phi_R(\partial_t\theta_n-\partial_t\theta_m)}{H^{s-2}} &\leq \norm{\phi_R(u_n-u_m)\cdot(\phi_{2R}\nabla \theta_m)}{H^{s-2}}\\ &\hspace{-1cm}+ \norm{\phi_Ru_n\cdot \phi_{2R}\nabla(\theta_n-\theta_m)}{H^{s-2}}\\
&\hspace{-2.5cm}\leq 
\norm{\phi_R(u_n-u_m)}{H^{s-2}}\norm{\phi_{2R}\nabla\theta_m}{L^{\infty}}+\norm{\phi_R(u_n-u_m)}{L^{\infty}}\norm{\phi_{2R}\nabla\theta_m}{H^{s-2}}\\
&\hspace{-2.5cm}+\norm{\phi_Ru_n}{H^{s-2}}\norm{\phi_{2R}\nabla(\theta_n-\theta_m)}{L^{\infty}}+\norm{\phi_Ru_n}{L^{\infty}}\norm{\phi_{2R}\nabla(\theta_n-\theta_m)}{H^{s-2}}.
\end{align*}
Then since $\norm{\phi_{2R}\nabla\theta_m}{L^{\infty}}, \norm{\phi_{2R}\nabla\theta_m}{H^{s-2}}, \norm{\phi_Ru_n}{H^{s-2}}$ and $\norm{\phi_Ru_n}{L^{\infty}}$ are uniformly bounded in $n$ on $[0,T]$, and since $(\phi_Ru_n)$ and $(\phi_R\theta_n)$ are Cauchy in $C([0,T];H^{s-\beta}(\R^2))$  and $C([0,T];H^{s-1}(\R^2))$ respectively, we conclude that 
\begin{equation}
\label{Hsul eqn8}
\lim_{N\to \infty} \sup_{n,m\geq N} \norm{\phi_R(\partial_t\theta_n-\partial_t\theta_m)}{H^{s-2}}=0.
\end{equation}
Thus $(\phi_{R}\partial_t\theta_n)$ is Cauchy in $C([0,T];H^{s-2}(\R^2))$. Recall that $\phi_R\theta_n \to \phi_R\theta$ in $C([0,T]\times \R^2)$, so we also have $\phi_{R}\theta_n \to \phi_R\theta$ in $\mathcal{D}'([0,T]\times \R^2)$ which implies that $\phi_{R}\partial_t\theta_n \to \phi_{R}\partial_t\theta$ in $\mathcal{D}'([0,T]\times \R^2)$. By the uniqueness of limits, this allows us to conclude that for all $R>0$, $\phi_{R}\partial_t\theta_n \to \phi_{R}\partial_t\theta$ in $C([0,T];H^{s-2}(\R^2))$. Recall that for all $n\in \mathbb{N}$ one has 
\begin{equation}
\label{Hsul eqn8}
\phi_{R}\partial_t\theta_n = -\phi_R u_n \cdot (\phi_{2R}\nabla \theta_n).
\end{equation}
Passing to the limit in \eqref{Hsul eqn8} in the $H^{s-2}$ norm allows us to conclude that $\phi_{R}\partial_t\theta=-\phi_{R} u \cdot \nabla \theta$ in $H^{s-2}(\R^2)$.
\\
\\
We next show that $\theta \in L^{\infty}([0,T];H_{ul}^s(\R^2))$. Recall that for each $z \in \R^2, n \in \mathbb{N}$, and $ t \in [0,T]$ one has 
\begin{equation}
\label{Hsul eq9}
\norm{\phi_z\theta_n(t)}{H^s} \leq C.
\end{equation}
Thus there is a subsequence of $(\phi_z\theta_n(t))$ depending on $z$ and $t$ (which we relabel as $(\phi_z\theta_n(t))$), that converges weak-* in $H^s(\R^2)$. We also have shown that for every $R>0$ and $t\in [0,T]$ one has $\phi_{0,R}\theta_n(t) \to \phi_{0,R}\theta(t)$ in $H^{s-1}(\R^2)$. Given $z$, we choose $R>0$ big enough so that $\phi_z=\phi_z\phi_{0,R}$, thus for this $R$ we have $\phi_z\theta_n(t) \to \phi_z\theta(t)$ in $H^{s-1}(\R^2)$. By uniqueness of limits, we also have that $\phi_z\theta_n(t)$ converges weak-* to $\phi_z\theta(t)$ in $H^s(\R^2)$ and $\theta$ satisfies the same bound as in \eqref{Hsul eq9}. Since this holds for all $z \in \R^2$ and $t\in [0,T]$ we see that $\theta \in L^{\infty}([0,T];H_{ul}^s(\R^2))$. Similarly, utilizing the uniform bound on $\norm{\phi_zu_n(t)}{H^{s+1-\beta}}$ and applying an identical argument as above allows us to conclude that $u \in L^{\infty}([0,T];H^{s+1-\beta}_{ul}(\R^2))$.
\\
\\
\textbf{$(u,\theta)$ satisfies \eqref{Sobolev Serfati Identity}}
To see that $(u,\theta)$ satisfy the Serfati-type identity 
$$
u(t)=u^0 + \nabla^{\perp}(a\Phi_{\beta})\ast (\theta(t)-\theta^0) - \int_0^t \nabla\nabla^{\perp}((1-a)\Phi_{\beta})\ast\cdot(\theta u)(\tau)d\tau,
$$
note that by the Sobolev Embedding Theorem $\theta \in C([0,T];C^{\alpha}(\R^2))$ and\\ $u \in C([0,T];C^{\alpha+1-\beta}(\R^2))$ for some $\alpha >1$, so applying Theorem \ref{Holder Result} we see that $(u,\theta)$ satisfies \eqref{Sobolev Serfati Identity}.
\\
\\
\textbf{Uniqueness.} Our last goal is to show the uniqueness of solutions to \eqref{gSQG}. Note that if $(u_1,\theta_1)$ and $(u_2,\theta_2)$ are two solutions generated from the same initial data $(u^0,\theta^0)$, if we set $\overline{\theta}=\theta^1-\theta^2$ and $\overline{u}=u^1-u^2$, then 

\begin{equation}\begin{cases}
\label{Hsul uniqueness pde}
\partial_t\overline{\theta}+u^1\cdot\nabla\overline{\theta}+\overline{u}\cdot\nabla\theta^2=0 \\
\overline{\theta}\vert_{t=0}=0,
\end{cases}
\end{equation}
and $\overline{u}$ satisfies the following identity for each $t\in [0,T]$:
\begin{equation}
\label{Hsul serfati uniqueness}
\overline{u}(t)=\nabla^{\perp}(a\Phi_{\beta})\ast\overline{\theta}(t)-\int_0^t\nabla\nabla^{\perp}((1-a)\Phi_{\beta})\ast\cdot(\overline{\theta}u^1+\theta^2\overline{u})(\tau)d\tau.
\end{equation}
We let $X(t,x)$ be the flow map for $u^1$, i.e. $X$ satisfies the following ODE:
\begin{equation}
\label{u1 flow map}
\begin{cases}
\partial_t X(t,x) = u^1(t,X(t,x))\\
X(x,0)=x.
\end{cases}
\end{equation}
We remark that since $\nabla u^1(t) \in L^{\infty}(\R^2)$ and $u^1$ is divergence free, this flow map exists and is measure preserving. Thus rewriting \eqref{Hsul uniqueness pde} in Lagrangian coordinates, integrating in time, and taking the $L^{\infty}$ norm yields
\begin{equation}
\label{estimate 1 for Hsul uniqueness}
\norm{\overline{\theta}(t)}{L^{\infty}} \leq \int_0^t \norm{\overline{u}(\tau)}{L^{\infty}}\norm{\nabla \theta^2(\tau)}{L^{\infty}}d\tau.
\end{equation}
Then taking the $L^{\infty}$ norm of \eqref{Hsul serfati uniqueness} we see that for all $t \in [0,T]$
\begin{equation}
\label{estimate 2 for Hsul uniqueness}
\norm{\overline{u}(t)}{L^{\infty}} \lesssim \norm{\overline{\theta}(t)}{L^{\infty}}+ \int_0^t \left(\norm{u^1(\tau)}{L^{\infty}}\norm{\overline{\theta}(\tau)}{L^{\infty}} + \norm{\theta^2(\tau)}{L^{\infty}}\norm{\overline{u}(\tau)}{L^{\infty}}\right)d\tau.
\end{equation}
Substituting \eqref{estimate 1 for Hsul uniqueness} into \eqref{estimate 2 for Hsul uniqueness}, and then adding the two quantities we see that 
$$
\norm{\overline{\theta}(t)}{L^{\infty}}+\norm{\overline{u}(t)}{L^{\infty}} \lesssim \int_0^t \left( \norm{\overline{u}(\tau)}{L^{\infty}}\norm{ \theta^2(\tau)}{\tilde{C}^1} + \norm{u^1(\tau)}{L^{\infty}}\norm{\overline{\theta}(\tau)}{L^{\infty}}\right)d\tau.
$$
Utilizing the uniform bounds of $\norm{\theta^2(t)}{\tilde{C}^1}$ and $\norm{u^1(\tau)}{L^{\infty}}$, and applying Grönwall's Lemma shows uniqueness.
\end{proof}

\appendix
\section{\\Estimates in Hölder Spaces}

\begin{lemma} If $r>0$ and $f \in C^r(\R^2)$, then
$$
\norm{\nabla^{\perp}(a\Phi_{\beta})\ast f}{L^{\infty}} \lesssim \norm{f}{C^r}.
$$
\end{lemma}
\begin{proof} First, by the Leibniz rule we note
$$
\norm{\nabla^{\perp}(a\Phi_{\beta})}{L^1} \lesssim \norm{\Phi_{\beta}}{L^1(\supp{a})}\norm{\nabla a}{L^{\infty}} + \norm{\nabla \Phi_{\beta}}{L^1(\supp{a})}\norm{a}{L^{\infty}}.
$$
Since $\beta \in (0,1)$,  $\Phi_{\beta}\in W^{1,1}(\supp{a})$. Applying Young's convolution inequality we have that 
$$
\norm{\nabla^{\perp}(a\Phi_{\beta})\ast f}{L^{\infty}} \lesssim \norm{f}{L^{\infty}} \lesssim \norm{f}{C^r}.
$$
\end{proof}

\begin{lemma}
\label{Lemma A2}
Let $j \in \mathbb{Z}$. Suppose that $f \in L^p(\R^d)$ for some $p \in [1,\infty]$, then 
\begin{equation}
\label{berstein type inequality}
\norm{\dot{\Delta}_j \nabla^{\perp}(-\Delta)^{-1+\beta/2}f}{L^p} \lesssim 2^{j(-1+\beta)}\norm{\dot{\Delta}_jf}{L^p}
\end{equation}
\end{lemma}
\begin{proof} By Bernstein's lemma \eqref{berstein type inequality} is equivalent to showing 
$$
\norm{\dot{\Delta}_j\nabla^{\perp}(-\Delta)^{-1+\frac{\beta}{2}}f}{L^p} \lesssim \norm{\dot{\Delta}_j(-\Delta)^{\frac{\beta-1}{2}}f}{L^p}.
$$
We see that 
\begin{align*}
\norm{\dot{\Delta}_j\nabla^{\perp}(-\Delta)^{-1+\frac{\beta}{2}}f}{L^p} &= \norm{\mathcal{F}^{-1}\left(\varphi_j(\xi)\frac{\xi^{\perp}\abs{\xi}}{\abs{\xi}^2}\abs{\xi}^{-1+\beta}\hat{f}(\xi)\right)}{L^p}\\
&= \norm{\dot{\Delta}_j(\check{m}\ast (-\Delta)^{\frac{\beta-1}{2}}f)}{L^p},
\end{align*}
where we define
$$
m(\xi) \eqdef\frac{\xi^{\perp}\abs{\xi}}{\abs{\xi}^2}.
$$
Observe that
$$
m\vert_{\supp{\varphi_j}}(\xi) =\chi(2^{-3-j}\xi)(1-\chi(2^{-j+1}\xi))m(\xi)\eqdef h_j(\xi),
$$
where $h_j(\xi)=h(2^{-j}\xi)$, and 
$$
h(\xi) = \chi(2^{-3}\xi)(1-\chi(2\xi))m(\xi).
$$
Thus $\check{h}_j(x)=2^{2j}\check{h}(2^jx)$. By a change of variables we see that $\norm{\check{h}_j}{L^1}=\norm{\check{h}}{L^1}=C<\infty$. Then, applying Young's convolution inequality 
\begin{align*}
\norm{\dot{\Delta}_j\nabla^{\perp}(-\Delta)^{-1+\frac{\beta}{2}}f}{L^p} &= \norm{\check{h} \ast \dot{\Delta}_j (-\Delta)^{\frac{\beta-1}{2}}f}{L^p}\\
&\leq 
\norm{\check{h}}{L^1}\norm{\dot{\Delta}_j (-\Delta)^{\frac{\beta-1}{2}}f}{L^p}\\
&=C\norm{\dot{\Delta}_j (-\Delta)^{\frac{\beta-1}{2}}f}{L^p}.
\end{align*}
\end{proof}
We will make use of the following lemma when proving estimates on \eqref{gSQG} in the Hölder-Zygmund spaces.
\begin{lemma}
\label{Velocity Holder Estimate}
Let $r>1$. If for every $j\geq 0$, $f\in L^{\infty}(\R^d)$ and $g \in C^r(\R^d)$ satisfy
$$
\Delta_j f = \Delta_j\nabla^{\perp}(-\Delta)^{-1+\frac{\beta}{2}} g
$$
almost everywhere on $\R^d$, then $f \in C^{r+1-\beta}(\R^d)$ with
$$
\norm{f}{C^{r+1-\beta}} \lesssim \norm{f}{L^{\infty}}+\norm{g}{C^r}.
$$
\end{lemma}
\begin{proof}
We have that 
\begin{align*}
\norm{f}{C^{r+1-\beta}} &= \sup_{j \geq -1} 2^{j(r+1-\beta)}\norm{\Delta_j f}{L^{\infty}}\\
&\lesssim 
\norm{\Delta_{-1}f}{L^{\infty}}+\sup_{j \geq 0} 2^{j(r+1-\beta)}\norm{\Delta_j f}{L^{\infty}}\\
&\lesssim
\norm{f}{L^{\infty}}+\sup_{j \geq 0} 2^{j(r+1-\beta)}\norm{\Delta_j \nabla^{\perp}(-\Delta)^{-1+\frac{\beta}{2}}g}{L^{\infty}}\\
&\lesssim
\norm{f}{L^{\infty}}+\sup_{j \geq 0} 2^{j(r+1-\beta)}2^{j(\beta-1)}\norm{\Delta_j g}{L^{\infty}}\\
&\leq 
\norm{f}{L^{\infty}}+\sup_{j \geq -1} 2^{jr}\norm{\Delta_j g}{L^{\infty}}\\
&=
\norm{f}{L^{\infty}}+\norm{g}{C^r}.
\end{align*}
In the second inequality, we used that for $j\geq 0$ one has $\Delta_j f = \Delta_j\nabla^{\perp}(-\Delta)^{-1+\frac{\beta}{2}} g$ almost everywhere on $\R^d$, and in the third inequality we used Lemma \ref{Lemma A2}.
\end{proof}

We have the following commutator estimate due to \cite{5}.
\begin{proposition}(Comuutator Estimate in Hölder Spaces)
\label{Holder Commutator Estimate}Let $j \geq -1$ and $r>0$, and let 
$$[u\cdot\nabla,\Delta_j]\theta=u\cdot\nabla(\Delta_j\theta)-\Delta_j(u\cdot\nabla\theta).$$
\begin{enumerate}
\item If $u,\theta \in \tilde{C}^1 \cap C^r$, then
\begin{equation*}
\norm{[u\cdot\nabla,\Delta_j]\theta}{C^r} \leq C\left(\norm{\nabla \theta}{L^{\infty}}\norm{u}{C^r} + \norm{\nabla u}{L^{\infty}}\norm{\theta}{C^r}\right).
\end{equation*}

\item If $\theta \in C^r$ and $u \in C^{r+1}$, then
\begin{equation*}
\norm{[u\cdot\nabla,\Delta_j]\theta}{C^r} \leq C\left(\norm{\theta}{L^{\infty}}\norm{u}{C^{r+1}} + \norm{\nabla u}{L^{\infty}}\norm{\theta}{C^r}\right).
\end{equation*}
\end{enumerate}
\end{proposition}

\section{\\Hölder Space Result}

For smooth decaying solutions to \eqref{gSQG} we have the constitutive law given by $u=\nabla^{\perp}(-\Delta)^{-1+\frac{\beta}{2}}\theta$. If $\theta$ lacks spatial decay, this expression ceases to  make sense. Instead, we make use of the following Serfati-type identity. The proofs of the following two lemmas are identical to those in \cite{1} so we omit them.
\begin{lemma}
\label{Serfati Identity}
Suppose that $(u,\theta)$ are smooth solutions to \eqref{gSQG} on $[0,T]\times\R^2$ with $\supp{\theta}\subset \R^2$ compact. Then for all $t \in [0,T]$ we have that
$$
u(t)=u^0 + \nabla^{\perp}(a\Phi_{\beta})\ast (\theta(t)-\theta^0) - \int_0^t \nabla \nabla^{\perp}((1-a)\Phi_{\beta})\ast\cdot (\theta u)(s)ds.
$$
\end{lemma}
$$
$$
\begin{lemma}
\label{Holder Space Frequency Equality} Suppose that $(u_n,\theta_n)_n$ are as in Theorem \ref{Holder Result}, then for any $j \in \mathbb{Z}$ and any $n \in \mathbb{N}$ we have that
$$
\dot{\Delta}_j u^n = \dot{\Delta}_j \nabla^{\perp}(-\Delta)^{-1+\frac{\beta}{2}}\theta^n.
$$
\end{lemma}
We are now set up to show well-posedness of \eqref{gSQG} in Hölder spaces.
\begin{theorem}
\label{Holder Result}
For $r \in (1,\infty)$ let $\theta^0 \in C^r(\R^2)$ and $u^0 \in L^{\infty}(\R^2)$ be such that
$$
u^0 = \nabla^{\perp}(-\Delta)^{-1+\frac{\beta}{2}} \theta^0 \text{ in } \dot{C}^r.
$$
Then there exists a time $T>0$ and a unique solution to 
\begin{equation}
\label{Holder gSQG}
\begin{cases}
\partial_t \theta + u\cdot \nabla \theta=0\\
(u,\theta)\vert_{t=0}=(u^0,\theta^0),
\end{cases}
\end{equation}
satisfying for any $r'<r$
\begin{align*}
\theta &\in L^{\infty}([0,T]; C^r(\R^2)) \cap \text{Lip}([0,T]; C^{r-1}(\R^2)) \cap C([0,T]; C^{r'}(\R^2)),\\
u &\in L^{\infty}([0,T]; C^{r+1-\beta}(\R^2)) \cap C([0,T]; C^{r'+1-\beta}(\R^2)).
\end{align*}
Furthermore, there exists some constant $C>0$ such that $(u,\theta)$ satisfies
\begin{equation}
\label{bound on holder norms of solution}
\norm{u}{L^{\infty}([0,T];L^{\infty})}+\norm{\theta}{L^{\infty}([0,T];C^r)} \leq \frac{C(\norm{u^0}{L^{\infty}}+\norm{\theta^0}{C^r})}{1-CT(\norm{u^0}{L^{\infty}}+\norm{\theta^0}{C^r})},
\end{equation}
and $(u,\theta)$ satisfies the Serfati-type identity, as in Lemma \ref{Serfati Identity}:
$$
u(t)=u_0 + \nabla^{\perp}(a\Phi_{\beta})\ast (\theta(t)-\theta_0) - \int_0^t \nabla \nabla^{\perp}((1-a)\Phi_{\beta})\ast\cdot (\theta u)(s)ds,
$$
for each $t \in [0,T]$.
\end{theorem}
 \begin{proof}
 \textbf{Creating an Approximating Sequence}: We follow \cite{1} exactly. We begin by defining sequences $\{\theta^n\}_{n=1}^{\infty}$ and $\{u^n\}_{n=1}^{\infty}$ as follows: we initialize them by setting
 \begin{equation}
 \label{initial holder sequence}
 \theta^1(t,x) = S_2\theta^0(x) \text{ and } u^1(t,x)=S_2u^0(x).
 \end{equation}
Then, for $n \geq 1$ we obtain $\theta^{n+1}$ from $u^n$ by solving
\begin{equation}
\label{Holder recover theta}
\begin{cases}
\partial_t \theta^{n+1} + u^n \cdot \nabla \theta^{n+1}=0\\
\theta^{n+1}(x,0)=S_{n+2}\theta^0.\\
\end{cases}
\end{equation}
We recover the velocity $u^{n+1}$ from $\theta^{n+1}$ 
by solving the following integral equation
\begin{align}
\label{Holder recover velocity}
u^{n+1}(t)&=S_{n+2}u^0 + \nabla^{\perp}(a\Phi_{\beta}) \ast (\theta^{n+1}(t)-\theta^{n+1}(0))\\
&- \int_0^t \nabla \nabla^{\perp} ((1-a)\Phi_{\beta}) \ast \cdot (\theta^{n+1}u^n)(\tau)d\tau \nonumber.
\end{align}
\\
\textbf{Uniform Bounds:}
Observe that by standard techniques \eqref{Holder recover theta}  has a unique solution in $B_R \subset L^{\infty}([0,T_n]; C^r(\R^2))$ where $R=2C(\norm{\theta^0}{C^r}+\norm{u^0}{L^{\infty}})$. Our goal is to establish uniform bounds on $\norm{u_n(t)}{L^{\infty}}$ and $\norm{\theta_n(t)}{C^r}$. Note by Proposition \ref{Velocity Holder Estimate} such bounds will imply a uniform bound on $\norm{u_n(t)}{C^{r+1-\beta}}$.
We aim to find a time $T>0$ and constant $M>0$ such that 
\begin{equation}
\label{Induction for Holder Spaces}
\norm{u^n(t)}{L^{\infty}} + \norm{\theta^n(t)}{C^r} \leq M \text{ for all } t\in [0,T].
\end{equation}
We will do this by using induction.
For the base case $n=1$ we have, by \eqref{initial holder sequence}
$$
\norm{u^1(t)}{L^{\infty}} + \norm{\theta^1(t)}{C^r} \leq C\left(\norm{u^0}{L^{\infty}}+\norm{\theta^0}{C^r}\right).
$$
Set $M=2C(\norm{\theta^0}{C^r} + \norm{u^0}{L^{\infty}})$ and choose a time $T>0$ such that 
$
\exp(CTM)\leq 2.
$ One can explicitly solve for $T$ and see that
\begin{equation}
\label{holder time}
T \leq  \frac{C}{\norm{u^0}{L^{\infty}}+\norm{\theta^0}{C^r}}.
\end{equation}
Hence,
$$
\norm{\theta^{1}(t)}{C^r} + \norm{u^{1}(t)}{L^{\infty}} \leq C(\norm{\theta^0}{C^r} + \norm{u^0}{L^{\infty}})\exp{(C_rTM)}\leq M
$$
which establishes the base case. We next assume that \eqref{Induction for Holder Spaces} is true for $n=1,2,\hdots,k$ for some fixed positive integer $k$. We then want to show that \eqref{Induction for Holder Spaces} is true for $n=k+1$. We see that 
\begin{equation}
\label{theta k+1 inductive holder bound}
\norm{\theta^{k+1}(t)}{C^r} \leq C\norm{\theta^0}{C^r}+C\int_0^t \norm{\theta^{k+1}(\tau)}{C^r}\norm{u^k(\tau)}{C^r}d\tau.
\end{equation}
Then, estimating the velocity we obtain
\begin{equation}
\label{velocity k+1 inductive holder bound}
\norm{u^{k+1}(t)}{L^{\infty}}\leq C(\norm{u^0}{L^{\infty}}+\norm{\theta^0}{C^r})+\int_0^t \norm{\theta^{k+1}(\tau)}{L^{\infty}}\norm{u^k(\tau)}{L^{\infty}}d\tau.
\end{equation}
Adding \eqref{theta k+1 inductive holder bound} to \eqref{velocity k+1 inductive holder bound}, utilizing the inductive hypothesis, and applying Grönwall's Lemma gives
\begin{equation}
\label{Holder uniform bounds}
\norm{\theta^{k+1}(t)}{C^r}+\norm{u^{k+1}(t)}{L^{\infty}} \leq C(\norm{\theta^0}{L^{\infty}}+\norm{u^0}{L^{\infty}})\exp{(CMT)}\leq M.
\end{equation}
Thus $\{u^n\}_{n=1}^{\infty}$ and $\{\theta^n\}_{n=1}^{\infty}$ are uniformly bounded in $L^{\infty}([0,T];L^{\infty}(\R^2))$ and in $L^{\infty}([0,T]; C^r(\R^2))$, respectively. 
Then for each $n \in \mathbb{N}$ we have
\begin{align*}
\norm{\partial_t\theta^n(t)}{C^{r-1}} &= \norm{u^{n-1}\cdot\nabla\theta ^n}{C^{r-1}}\\
&\leq \norm{u^{n-1}}{C^{r-1}}\norm{\nabla \theta^{n}}{L^{\infty}} + \norm{u^{n-1}}{L^{\infty}}\norm{\theta^{n}}{C^r}\\
&\leq C \norm{\theta^n}{C^r}\norm{u^{n-1}}{C^{r-1}} \leq CM^2.
\end{align*}
Therefore $\partial_t \theta^n \in L^{\infty}([0,T], C^{r-1}(\R^2))$ which by the Mean-Value Theorem implies that $\theta^n \in Lip([0,T]; C^{r-1}(\R^2))$. Our next goal is to show that $\{\theta^n\}$ is Cauchy in $C([0,T];C^{r-1}(\R^2))$ and that $\{u^n\}$ is Cauchy in $C([0,T];L^{\infty}(\R^2))$. We define the quantities
$$
\eta^n=\theta^n-\theta^{n-1} \text{ and } v^n=u^n-u^{n-1}.
$$
Then by our definition of the approximation sequence we see that 
\begin{align*} \eta^1 &= S_2\theta^0 - \theta^0\\
v^1 &= S_2u^0-u^0,
\end{align*}
and for any $n \geq 1$ we have that 
\begin{equation}
\label{cauchy pde in Holder space}
\begin{cases}
\partial_t \eta^{n+1}+u^n\cdot \nabla \eta^{n+1} = -v^n\cdot \nabla \theta^n \hspace{.5cm} \text{in } [0,T] \times \R^2\\
\eta^{n+1}(x,0) = \Delta_{n+2}\theta^0(x) \hspace{.5cm} \text{in } \R^2\\
v^{n+1}(x,0)=\Delta_{n+2}u^0(x) \hspace{.5cm} \text{in } \R^2.
\end{cases}
\end{equation}
Note that $v^n$ satisfies the identity
\begin{align}
\label{cauchy serfati identity in Holder space}
v^n(t) - v^n(0) &= \nabla^{\perp}(a\Phi_{\beta})\ast (\eta^n(t) -\eta^n(0))\nonumber \\ &- \int_0^t \nabla\nabla^{\perp}((1-a)\Phi_{\beta}) \ast\cdot(\eta^nu^{n-1}+\theta^{n-1}v^{n-1})(\tau)d\tau.
\end{align}
Integrating \eqref{cauchy pde in Holder space}$_1$ in time, switching to Lagrangian coordinates,  applying the Littlewood-Paley operator $\Delta_j$ for $j \geq -1$, introducing a commutator, and then taking the $L^{\infty}$ norm yields the following estimate
\begin{align*}
\norm{\Delta_j\eta^{n+1}(t)}{L^{\infty}} &\leq \norm{\Delta_j\eta^{n+1}(0)}{L^{\infty}} +\int_0^t \norm{[u^n\cdot\nabla,\Delta_j]\eta^{n+1}(\tau)}{L^{\infty}}d\tau\\
&+
\int_0^t \norm{\Delta_j(v^n\cdot\nabla\theta^n)(\tau)}{L^{\infty}}d\tau.
\end{align*}
Multiplying everything by $2^{j(r-1)}$, taking the supremum over $j \geq -1$, and applying Lemma \ref{Holder Commutator Estimate} to the first integral, and noting that for any $t\in [0,T]$
$$
\norm{\nabla u^n(t)}{L^{\infty}}, \norm{u^n(t)}{C^r}, \norm{\theta^n(t)}{C^r} \leq M 
$$
yields the following estimate:
\begin{align}
\label{Holder cauchy estimate 1}
\norm{\eta^{n+1}(t)}{C^{r-1}} &\leq \norm{\eta^{n+1}(0)}{C^{r-1}} \nonumber \\&+ M\int_0^t \left(\norm{\eta^{n+1}(\tau)}{C^{r-1}} +\norm{v^n(\tau)}{C^{r-1}}\right)d\tau.
\end{align}
Next, taking the $L^{\infty}$ norm of \eqref{cauchy serfati identity in Holder space} and using the uniform bounds on $\norm{u^n(t)}{L^{\infty}} \text{ and } \norm{\theta^n(t)}{L^{\infty}}$ gives
\begin{align}
\label{Holder cauchy estimate 2}
\norm{v^{n+1}(t)}{L^{\infty}} &\lesssim \norm{v^{n+1}(0)}{L^{\infty}}+\norm{\eta^{n+1}(0)}{L^{\infty}}+\norm{\eta^{n+1}(t)}{L^{\infty}} \nonumber\\&+ M \int_0^t \left(\norm{\eta^{n+1}(\tau)}{L^{\infty}}+\norm{v^n(\tau)}{L^{\infty}}\right)d\tau.
\end{align}
Then adding \eqref{Holder cauchy estimate 1} to \eqref{Holder cauchy estimate 2} allows us to conclude 
\begin{align}
&\norm{\eta^{n+1}(t)}{C^{r-1}} + \norm{v^{n+1}(t)}{L^{\infty}}
\lesssim \norm{\eta^{n+1}(0)}{C^{r-1}} + \norm{v^{n+1}(0)}{L^{\infty}}\nonumber \\
&\hspace{1cm}+ M \int_0^t \left( \norm{\eta^{n+1}(\tau)}{C^{r-1}} + \norm{\eta^{n+1}(\tau)}{L^{\infty}} + \norm{v^{n}(\tau)}{C^{r-1}} + \norm{v^n(\tau)}{L^{\infty}} \right)d\tau.
\end{align}
By Lemma \ref{Holder Space Frequency Equality} for all $j \in \mathbb{Z}$ 
$$
\dot{\Delta}_j v^n = \dot{\Delta}_j \nabla^{\perp}(-\Delta)^{-1+\frac{\beta}{2}}\eta^{n},
$$
which in turn implies that 
Note $\norm{v^{n}(t)}{C^{r-1}}\lesssim\norm{v^n(t)}{C^{r-\beta}} \leq \norm{v^n(t)}{L^{\infty}}+ \norm{\eta^n(t)}{C^{r-1}}$. We then bound 
\begin{multline}
\int_0^t \left( \norm{\eta^{n+1}(\tau)}{C^{r-1}} + \norm{\eta^{n+1}(\tau)}{L^{\infty}} + \norm{v^{n}(\tau)}{C^{r-1}} + \norm{v^n(\tau)}{L^{\infty}} \right)d\tau \\
\lesssim 
\int_0^t \left( \norm{\eta^{n+1}(\tau)}{C^{r-1}} + \norm{v^{n+1}(\tau)}{L^{\infty}} \norm{v^n(\tau)}{L^{\infty}} + \norm{\eta^n(\tau)}{C^{r-1}}\right)d\tau.
\end{multline}
We remark that the term $\norm{v^{n+1}(\tau)}{L^{\infty}}$ was artificially introduced for the proceeding argument. We set $D_n(t) = \norm{v^n(t)}{L^{\infty}}+\norm{\eta^n(t)}{C^{r-1}}$. 
Then it follows that
\begin{equation}
\label{pre sum estimate Holder}
D_{n+1}(t) \lesssim_r D_{n+1}(0) + M\int_0^t (D_{n+1}(\tau)+D_n(\tau))d\tau.
\end{equation}
We define the function $E(t)$ as follows:
$$
E(t) \eqdef  \sum_{n=0}^{\infty} D_{n+1}(t).
$$
We claim that $E(0)$ is finite. This is since
\begin{align*}
E(0)&=\sum_{n=1}^{\infty}\left( \norm{\Delta_{n+2}u^0}{L^{\infty}} + \norm{\Delta_{n+2}\theta^0}{C^{r-1}}\right) \\
&\lesssim \sum_{n=1}^{\infty}\left(2^{-rn}\norm{u^0}{C^r} + 2^{-n}\norm{\theta^0}{C^r}\right)\\
&\lesssim \norm{u^0}{C^r}+\norm{\theta^0}{C^r}.
\end{align*}
Then summing \eqref{pre sum estimate Holder} and noting that
$$
\sum_{n=0}^{\infty} \left(D_{n+1}(t)+D_n(t)\right) = 2E(t) + D_0(t),
$$
then gives 
$$
E(t) \lesssim_r E(0)+MT + \int_0^t ME(\tau)d\tau,
$$
Applying Grönwall's inequality we find that 
$$
E(t) \lesssim_r (E(0)+MT)\exp(CMT) <\infty \text{ for all } t\in [0,T],
$$
thus $\lim_{n \to \infty} D_n(t) \to 0$ for any fixed $t \in [0,T]$. So we have shown  $$\lim_{n \to \infty} \norm{u^n(t)-u^{n-1}(t)}{L^{\infty}}=0 \text{ and } \lim_{n \to \infty}\norm{\theta^{n}(t)-\theta^{n-1}(t)}{C^{r-1}}=0.$$ 
Thus $\{u^n(t)\}$ and $\{\theta^n(t)\}$ are Cauchy in $L^{\infty}(\R^2)$ and $C^{r-1}(\R^2)$ respectively for any fixed $t \in [0,T]$. Then, following an identical argument to that in \cite{1}, we find that  $\theta^n$ is Cauchy in $C([0,T];C^{r-1}(\R^2))$ and $u^n$ is Cauchy in $C([0,T];L^{\infty}(\R^2))$.
It follows that $\{u^n\}$ converges to some $u$ in $C([0,T];L^{\infty}(\R^2))$ and $\{\theta^n\}$ converges to some $\theta$ in $C([0,T];C^{r-1}(\R^2))$. Since $\{u^n\}$ converges to $u$ we have that for every $j \in \mathbb{Z}$ 
$$
\norm{\dot{\Delta}_j(u^n-u)}{L^{\infty}} \leq \norm{u_n-u}{L^{\infty}}  \to 0 \text{ as } n \to \infty.
$$
Since $\theta^n$ is Cauchy is $C([0,T];C^{r-1}(\R^2))$ by Lemma \ref{Lemma A2}
$$
\norm{\dot{\Delta}_j\nabla^{\perp}(-\Delta)^{-1+\frac{\beta}{2}}(\theta^n-\theta)}{L^{\infty}}\lesssim 2^{j(\beta-1)}\norm{\theta^n-\theta}{L^{\infty}} \to 0 \text{ as } n \to \infty.
$$
Thus, we are able to pass to the limit in Lemma \ref{Holder Space Frequency Equality} to conclude that 
\begin{equation}
\label{frequency equality in Holder space}
\dot{\Delta}_j u = \dot{\Delta}_j \nabla^{\perp} (-\Delta)^{-1+\frac{\beta}{2}}\theta. 
\end{equation}
Furthermore, one has
$$
\theta \in L^{\infty}([0,T]; C^r(\R^2)) \cap Lip([0,T]; C^{r-1}(\R^2)).
$$
We see that $\theta \in L^{\infty}([0,T]; C^r(\R^2))$ since 
$$
\sup_{j \geq -1} 2^{jr}\norm{\Delta_j \theta}{L^{\infty}} \leq \lim_{n \to \infty} \sup_{j \geq -1} 2^{jr}\norm{\Delta_j \theta^n}{L^{\infty}}<\infty.
$$
Then applying Lemma \ref{Velocity Holder Estimate} to \eqref{frequency equality in Holder space} we conclude that
$$
u \in L^{\infty}([0,T]; C^{r+1-\beta}(\R^2)).
$$
Utilizing the uniform bounds \eqref{Holder uniform bounds}, interpolation between $C^{r-1}(\R^2)$ and $C^r(\R^2)$ shows that $\theta^n \to \theta$ in $C([0,T];C^{r'}(\R^2))$ for $r' \in [r-1,r)$. Similarly, interpolation between $L^{\infty}(\R^2)$ and $C^{r+1-\beta}(\R^2)$ shows that $u^n \to u$ in $C([0,T];C^{\alpha}(\R^2))$ for $\alpha \in (0,r+1-\beta)$. Thus for any $\alpha \in (0,r+1-\beta)$ and $ r' \in [r-1,r)$ we have shown that 
\begin{align*}
    \theta &\in L^{\infty}([0,T]; C^r(\R^2) \cap Lip([0,T]; C^{r-1}(\R^2)) \cap C([0,T];C^{r'}), \\
    u & \in L^{\infty}([0,T]; C^{r+1-\beta}(\R^2)) \cap C([0,T]; C^{\alpha}),
\end{align*}
which is enough regularity to conclude that the pair $(u,\theta)$ satisfies \eqref{Holder gSQG}. We next show that \eqref{bound on holder norms of solution} is true. Again following in the footsteps of \cite{1}, we set $\psi_n(\tau)=\norm{u^n(\tau)}{L^{\infty}}+\norm{\theta^n(\tau)}{C^r}$. From \eqref{Holder cauchy estimate 1} and \eqref{Holder cauchy estimate 2} we see that
\begin{equation}
\label{psi n estimate}
\psi_n(\tau) \lesssim \psi_n(0)\exp{\left(C_r\int_0^{\tau}\psi_n(s)ds\right)},
\end{equation}
which by the Fundamental Theorem of Calculus implies
$$
\frac{-1}{C_r}\frac{d}{d\tau}\left(e^{-C_r\int_0^{\tau}\psi_n(s)ds}\right) \lesssim \psi_n(0).
$$
Then integrating the above inequality and rearranging the terms we find that 
$$
\exp{\left(C_r\int_0^{\tau} \psi_n(s)ds \right)} \lesssim \frac{1}{1-C_rt\psi_n(0)}.
$$
From \eqref{psi n estimate} we see that $\frac{\psi_n(\tau)}{\psi_n(0)}\lesssim \exp{\left(C_r\int_0^{\tau} \psi_n(s)ds \right)}$. Thus 
\begin{equation}
\label{holder bound n}
\norm{u^n(t)}{L^{\infty}}+\norm{\theta^n(t)}{C^r} \leq \frac{C(\norm{u^0}{L^{\infty}}+\norm{\theta^0}{C^r})}{1-Ct(\norm{u^0}{L^{\infty}}+\norm{\theta^0}{C^r})}.
\end{equation}
Since \eqref{holder bound n} holds for all $n \in \mathbb{N}$ and fixed $t \in [0,T]$, we take the limit as $n \to \infty$ of the left hand side and use continuity of the norm to obtain \eqref{bound on holder norms of solution}. We now show that Lemma \ref{Serfati Identity} holds. Recall we constructed our approximation sequence as
\begin{align}
\label{Holder recover velocity}
u^{n}(t)&=u^{n}(0) + \nabla^{\perp}(a\Phi_{\beta}) \ast (\theta^{n}(t)-\theta^{n}(0))\\
&- \int_0^t \nabla \nabla^{\perp} ((1-a)\Phi_{\beta}) \ast \cdot (\theta^{n}u^{n-1})(\tau)d\tau\nonumber .
\end{align}
As $n\to \infty$ we have seen $u^n(t) \to u(t)$ in $L^{\infty}$. It remains to show that the right hand side of \eqref{Holder recover velocity} converges to the Serfati-type identity with $(u^n,\theta^n)$ replaced by $(u,\theta)$. We take the difference of the two identities, take the $L^{\infty}$ norm, and apply Young's convolution inequality and Hölder's inequality to bound the difference by
\begin{align*}
&\norm{(u^n-u)(0)}{L^{\infty}} + \norm{\nabla^{\perp}(a\Phi_{\beta})}{L^1}(\norm{(\theta^n-\theta)(t)}{L^{\infty}} + \norm{(\theta^n-\theta)(0)}{L^{\infty}}) \\
&+\norm{\nabla\nabla^{\perp}((1-a)\Phi_{\beta})}{L^1}\int_0^t \left( \norm{u^n(\tau)}{L^{\infty}}\norm{(\theta^n-\theta)(\tau)}{L^{\infty}} + \norm{\theta(\tau)}{L^{\infty}}\norm{(u^{n-1}-u)(\tau)}{L^{\infty}}\right)d\tau.
\end{align*}
Since we have defined $u^n(0)=S_{n+1}u^0$ it is clear that $\norm{(u^n-u)(0)}{L^{\infty}} \to 0$ as $n \to \infty$. Also since $a \Phi_{\beta} \in W^{1,1}(\R^2)$ it is clear that 
$$
\norm{\nabla^{\perp}(a\Phi_{\beta})}{L^1}(\norm{(\theta^n-\theta)(t)}{L^{\infty}} + \norm{(\theta^n-\theta)(0)}{L^{\infty}}) \to 0 \text{ as } n \to \infty.
$$
Utilizing the uniform bounds of $u^n \in L^{\infty}([0,T] \times \R^2)$ and passing to the limit we find that
$$
\lim_{n \to \infty} \int_0^t \left( \norm{u^n(\tau)}{L^{\infty}}\norm{(\theta^n-\theta)(\tau)}{L^{\infty}} + \norm{\theta(\tau)}{L^{\infty}}\norm{(u^{n-1}-u)(\tau)}{L^{\infty}}\right)d\tau = 0.
$$
Thus we have shown that Lemma \ref{Serfati Identity} holds for the pair $(u,\theta)$. 
\\
\\
\textbf{Uniqueness:} We suppose that we have two solutions $(u^1,\theta^1)$ and $(u^2,\theta^2)$ to \eqref{Holder gSQG} generated from the same initial data. We set $\eta = \theta^1-\theta^2$ and $\overline{u}=u^1-u^2$. We then see that $\overline{u}$ and $\eta$ satisfy the following differential equation
\begin{equation}
\label{uniqueness PDE Holder}
\begin{cases}
\partial_t \eta +u^1\cdot \nabla \eta +\overline{u}\cdot\nabla\theta^2=0 \\
\eta\vert_{t=0}=0.
\end{cases}
\end{equation}
Moreover, we also note that $\overline{u}$ satisfies the following identity for any $t>0$
\begin{equation}
\label{holder serfati uniqueness}
\overline{u}(t) = \nabla^{\perp}(a\Phi_{\beta})\ast\eta(t)-\int_0^t \nabla\nabla^{\perp}((1-a)\Phi_{\beta})\ast\cdot(\eta u^1 + \theta^2 \overline{u})(\tau)d\tau.
\end{equation}
Applying the Littlewood-Paley operator $\Delta_j$ to $\eqref{uniqueness PDE Holder}_1$, introducing a commutator, rewriting in Lagrangian coordinates (using the flow map for $u^1$), and then integrating in time yields
$$
\norm{\Delta_j\eta(t)}{L^{\infty}} \leq \int_0^t \norm{[u^1\cdot\nabla,\Delta_j]\eta(\tau)}{L^{\infty}}d\tau + \int_0^t \norm{\Delta_j(\overline{u}\cdot \nabla \theta^2)(\tau)}{L^{\infty}}d\tau.
$$
Applying the same techniques as earlier when we showed $\{\theta^n\}_n$ is a Cauchy sequence in $C^{r-1}(\R^2)$ then implies that 
\begin{equation}
\label{Holder Uniqueness estimate 1}
\norm{\eta(t)}{C^{r-1}} \lesssim \int_0^t \left( \norm{u^1(\tau)}{C^r}\norm{\eta(\tau)}{C^{r-1}} + \norm{\overline{u}(\tau)}{C^{r-1}}\norm{\theta^2(\tau)}{C^r}\right)d\tau.
\end{equation}
Similarly, we also find that 
\begin{equation}
\label{Holder Uniqueness estimate 2}
\norm{\overline{u}(t)}{L^{\infty}} \lesssim \norm{\eta(t)}{C^{r-1}} + \int_0^t \left( \norm{(\eta u^1)(\tau)}{L^{\infty}} + \norm{(\theta^2\overline{u})(\tau)}{L^{\infty}}\right)d\tau.
\end{equation}
Then utilizing the regularity of $u^1,\theta^2$, adding \eqref{Holder Uniqueness estimate 1} to \eqref{Holder Uniqueness estimate 2}, substituting \eqref{Holder Uniqueness estimate 1} for $\norm{\eta(t)}{C^{r-1}}$ in \eqref{Holder Uniqueness estimate 2} and using $\norm{\overline{u}(\tau)}{L^{\infty}} \lesssim \norm{\overline{u}(\tau)}{C^{r-1}}$ yields
\begin{equation}
\label{Holder uniqueness estimate 3}
\norm{\eta(t)}{C^{r-1}} + \norm{\overline{u}(t)}{L^{\infty}} \lesssim \int_0^t \left(\norm{\eta(\tau)}{C^{r-1}} + \norm{\overline{u}(\tau)}{C^{r-1}}\right)d\tau.
\end{equation}
Since we have the equality 
$$
\dot{\Delta}_j \overline{u} = \dot{\Delta}_j \nabla^{\perp}(-\Delta)^{-1+\beta/2}\eta \hspace{.5cm} \text{for all } j \in \mathbb{Z},
$$
applying Lemma \ref{Velocity Holder Estimate}  we find that $\norm{\overline{u}(t)}{C^{r-1}} \lesssim  \norm{\overline{u}(t)}{L^{\infty}} + \norm{\eta(t)}{C^{r-1}}$. Substituting this into \eqref{Holder uniqueness estimate 3} yields
\begin{equation}
\label{Holder uniqueness estimate 4}
\norm{\eta(t)}{C^{r-1}} + \norm{\overline{u}(t)}{L^{\infty}} \lesssim \int_0^t \left(\norm{\eta(t)}{C^{r-1}} + \norm{\overline{u}(\tau)}{L^{\infty}}\right)d\tau.
\end{equation}
Thus applying Grönwall's inequality we find that $\eta(t), \overline{u}(t) \equiv 0$ for all $t\in [0,T]$ which shows that the solution to \eqref{Holder gSQG} is unique.
\end{proof}

\section*{Acknowledgements}
\noindent The author thanks Elaine Cozzi for useful discussions and suggesting this project to work on.

\end{document}